\documentclass[12pt]{article}
\usepackage{amssymb,amsmath,latexsym,amsthm,amsfonts}
\usepackage[english]{babel}
\usepackage[T1]{fontenc}
\usepackage{color,hyperref}
\usepackage{graphics}
\usepackage{graphicx}
\usepackage{rotating}

\usepackage{multirow}

\usepackage{mathtools, amsfonts, amssymb,amsthm}
\mathtoolsset{showonlyrefs}

\usepackage[colorinlistoftodos]{todonotes}

\newcommand{\espcond}[2]{E\mathopen{}\left\{#1\middle|#2\right\}}

\newtheorem{theorem}{Theorem}[section]

\newtheorem{proposition}[theorem]{Proposition}

\newtheorem{remark}{Remark}[section]
\newtheorem{example}{Example}[section]

\newtheorem{assumption}{Assumption}

\def\1g{1\hskip -3pt \mbox{l}}

\small\normalsize

    \title{Powers correlation analysis of non-stationary illiquid assets}

\author{
{\sc Valentin Patilea\footnote{CREST Ensai,
Campus de Ker-Lann,
51 Rue Blaise Pascal,
BP 37203, 35172 BRUZ Cedex
FRANCE; email: patilea@ensai.fr. This author acknowledges
support from the research program \emph{Mod\`eles et traitements math\'ematiques des donn\' ees en tr\`es grande dimension,} of Fondation du Risque and MEDIAMETRIE}
\qquad \qquad
Hamdi Ra\"{\i}ssi\footnote{Instituto de Estad\'{\i}stica, PUCV,
Errazuriz 2734, Valpara\'{\i}so, CHILE; email: hamdi.raissi@pucv.cl. This author  acknowledges the ANID funding Fondecyt 1201898.}
\qquad \qquad
}
}

\begin{document}

    \maketitle


    \abstract{In this paper, the higher order dynamics of individual illiquid stocks are investigated. We show that considering the classical powers correlation could lead to a spurious assessment of the volatility persistency or long memory volatility effects, if the zero returns probability is non-constant over time. In other words, the classical tools are not able to distinguish between long-run volatility effects, such as IGARCH, and the case where the zero returns are not evenly distributed over time. As a consequence, tools that are robust to changes in the degree of illiquidity are proposed. Since a time-varying zero returns probability could potentially be accompanied by a non-constant unconditional variance, we then develop powers correlations that are also robust in such a case. In addition, note that the tools proposed in the paper offer a rigorous analysis of the short-run volatility effects, while the use of the classical power correlations lead to doubtful conclusions. The Monte Carlo experiments, and the study of the absolute value correlation of some stocks taken from the Chilean financial market, suggest that the volatility effects are only short-run in many cases.}

\quad


\section{Introduction}
\label{intro}

The analysis of the returns powers serial correlation structures, is a routine task in quantitative finance. For instance, considering a daily stock return $r_t$, the lack of second order serial correlations ($corr(r_t^2,r_{t-h}^2)=0$, $h\neq0$), suggests no volatility clustering pattern. In particular, this would mean that the analyzed asset is not subject to volatility spillover from the market. Also, if we conclude that a given asset has a constant volatility, it can be classified as a safe haven portfolio component in times of crises, as it is not subject to higher risk levels. Moreover, when a relatively large $h$ is examined, $corr(r_t^2,r_{t-h}^2)=0$ means there is no shocks persistency, or long memory effects for the volatility. Conversely, if we find some linear dynamics in the returns powers, then a stochastic volatility is usually fitted. This means that the risk levels should be computed conditional on the date. See the algorithm in Francq and Zako\"{\i}an (2019, p. 339). Currently there are many papers where the autocorrelations of returns powers are examined. See Taylor (2007) and St\u{a}ric\u{a} and Granger (2005) among many others. Several tools for investigating the assets powers correlations are available in the literature, and widely used nowadays. The two prominent tests for analyzing the correlations of the squared returns are the ARCH-LM test, introduced by Engle (1982), and the portmanteau test, see McLeod and Li (1983). Since the two tests are equivalent, as shown for instance in Francq and Zako\"{\i}an (2019, p.148), throughout this paper we shall focus on the portmanteau test.

Although the above mentioned tools are widely used, their properties for illiquid assets are poorly investigated. We call an \emph{illiquid asset}, an asset with a daily price change probability less than 1, or equivalently with  a positive zero-return probability.
Stocks with many zero returns records are quite common in all financial markets. In particular, several authors have documented a large amount of illiquid stocks for emerging markets.  See, for instance, Lesmond (2005, p. 424--426) among many others. The main message  of this paper  is that a non-stationary probability of a daily price change, could spuriously lead to conclude that nonlinear effects are present in the series. Also, it can lead to a spurious assessment of the shocks persistence, or long memory effects for the volatility. To avoid such misleading analyses in the case of a time-varying daily price change probability, we provide adapted tools based on the power returns serial correlations. Several authors have underlined that non-stationary behaviors can generate so-called, stylized effects in finance. See, \emph{e.g.}, Mikosch and St\u{a}ric\u{a} (2004), Fry\'{z}lewicz (2005), Engle and Rangel (2008), Patilea and Ra\"{i}ssi (2014),  Ra\"{i}ssi (2018). In some sense, this paper complements the above references in the framework of illiquid assets.

Several patterns of non-constant daily price change probability can be commonly encountered in practice. In Figure \ref{ns-stocks}, representative examples taken from the Chilean stocks market are shown\footnote{The authors are grateful to Andres Celedon for research assistance.}. For the Cruzados stock, a long run decreasing trend can be observed for this probability. On the other hand, the Conchatoro stock shows an increasing probability of non zero returns, due to the fast developing Chilean economy during the period 2000-2008. Quick shifts can be suspected in many cases. The abrupt decrease of the daily price change probability in the bottom left panel of Figure \ref{ns-stocks}, is certainly the consequence of the take-over bid of Metlife on the pension funds administration company Provida during September 2013. On that occasion Metlife acquired more than 90\% of the share capital of Provida. Conversely, a rapid increase in the daily price change probability can be observed for the Molymet stock, due to the capital increase achieved in November 2010. Note that sometimes, stocks that seem to have a constant unconditional non-zero returns probability may be encountered (see the panel corresponding to the Lipigas and Las Condes stocks). However, it is important to underline that the standard tools, are not able to distinguish between a non-stationary probability of a daily price change, and the presence of nonlinear effects or unconditional heteroscedasticity. Indeed, a higher non-zero returns probability, can be, for instance, confused with a clustering of the volatility by the standard portmanteau test. As a consequence, a modified test taking into account a non-standard behavior of the non-zero returns probability will be proposed.

Our analysis is related to the so-called missing and amplitude modulation observations problems in time series. Indeed, a zero-return could be considered as a missing observation of a latent return. See, for instance, Parzen (1963), Dunsmuir and Robinson (1981), Stoffer and Toloi (1992) and the references therein. The main difference with respect to such contributions, comes from the fact that we allow the missing observations, the zero-returns in our case, to occur with a time-varying probability. We point out that this drastically changes the performance and behavior of standard tools, such as the portmanteau test.

The structure of this study is as follows. In Section \ref{main.res}, the detailed framework of our study is presented. The validity of the classical power correlations statistics when the zero returns probability is constant is underlined. On the other hand, we show that the standard tools are inadequate in the presence of a non-constant zero returns probability, or a time-varying unconditional variance. As a consequence, statistics which are robust to such commonly observed patterns are proposed in Section \ref{sec:corr} and \ref{sec:corr_adapt}. In Section \ref{num.exp}, Monte Carlo experiments are carried out, and the Taylor effect for the stocks presented in the Introduction is analyzed. Some concluding remarks are provided in Section \ref{concl.sec}.

\section{The framework of the study and validity of the classical powers correlations}
\label{main.res}

Consider a one-period, profit-and-loss, random latent continuous variable $\widetilde r_t$. Due to the illiquidity, some values of  $\widetilde r_t$ are replaced by zero. Let $r_1,\dots,r_n$ be the observed returns, with $n$ the sample size. To formalize the observation scheme we have in mind,
consider the binary process $(a_t)\subset \{0,1\}$. It is then assumed that  $r_t = a_t \widetilde r_t$. The process $(a_t)$, introduced by Parzen (1963), is commonly called the "amplitude modulated" sequence. The mechanism generating the $r_t$'s can be explained by a variety of facts, such as trading costs, or rounded prices. Throughout the paper, we focus on the case  $0<P(a_t=1)<1$.

When all the  $\widetilde r_t$ are available, that is when $P(a_t=1)=1$, the empirical serial correlations analysis of $|\widetilde r_t|^\delta$ is carried out for a variety of routine tasks. For instance, the serial correlations of powers returns are compared with adequate confidence bounds to determine some horizon where they seem to vanish. This may serve to evaluate the shocks persistency. Also, seasonality patterns can be identified (e.g. weekly market seasonality). Another important task would be to specify the volatility $\sigma_t$, if one assumes that
\begin{equation}\label{returns}
 \widetilde r_t=\sigma_t\eta_t,
\end{equation}
where $\sigma_t>0$, and the innovations process $(\eta_t)$ fulfills regularity conditions (see Assumption \ref{ident} below). More specifically, in quantitative finance, it is common to decide between the case where $\sigma_t=\sigma$ is constant, and the case where a stochastic $\sigma_t$ should be fitted to the data. Below, we describe in detail, two important examples, where the power returns serial correlations are widely used for modeling the volatility. In all the above tasks, when $P(a_t=1)=1$, the empirical power returns serial correlations
\begin{equation}\label{stat-std0}
\widehat{\Gamma}_0^{(\delta)} (m):=\left(\hat{\rho}_0^{(\delta)} (1),\dots,\hat{\rho}_0^{(\delta)} (m)\right)',\:\mbox{with}\:
\hat{\rho}_0^{(\delta)} (h):=\hat{\gamma}_0^{(\delta)} (h)\hat{\gamma}_0^{(\delta)} (0)^{-1},
\end{equation}
where
$
\hat{\gamma}_0^{(\delta)} (h)=n^{-1}\sum_{t=1+h}^{n}
\left(|\widetilde r_t|^\delta-\bar{\widetilde r}^{(\delta)}\right)\left(|\widetilde r_{t-h}|^\delta-\bar{\widetilde r}^{(\delta)} \right),
$
and $\bar{\widetilde r}^{(\delta)} =n^{-1}\sum_{t=1}^{n}|\widetilde r_t|^\delta$, are investigated.

\begin{example}[Testing for short-run dynamics in the stationary volatility]
\label{ex1}
Usually, the volatility is assumed stochastic strictly stationary (see Francq and Zako\"{\i}an (2019) and references therein), if the following null hypothesis is rejected: for some integer $m\geq 1$ and a given $\delta>0$,
\begin{equation}\label{hyp-test-2}
H_0^{(\delta)}:\Gamma_0^{(\delta)} (m)=0,
\end{equation}
where
\begin{equation}\label{Gamma_a=1}
\Gamma_0^{(\delta)}(m):=\left(\rho_0^{(\delta)}(1),\dots,\rho_0^{(\delta)}(m)\right)',
\end{equation}
with
$$
\rho_0^{(\delta)}(h) : = corr(|\widetilde r_t|^{(\delta)},| \widetilde r_{t-h} |^{(\delta)}):=\gamma_0^{(\delta)}(h)\gamma_0^{(\delta)}(0)^{-1},
$$
and
$$
\gamma_0^{(\delta)} (h) = E\left\{\left(|\widetilde r_t|^\delta-E(| \widetilde r_t|^\delta)\right)\left(| \widetilde r_{t-h} |^\delta-E(|\widetilde r_{t-h}|^\delta)\right)\right\}.
$$
In others words, a significant short-run empirical correlation for the sequence $|\widetilde r_t|^\delta$, 
indicates that the series displays the so-called volatility clustering pattern. Standard tools could be used to check  $H_0^{(\delta)}$ against $H_1^{(\delta)}:\Gamma_0^{(\delta)} (m) \neq 0$.  For instance, if $\delta=2$ one can use the portmanteau test of  McLeod and Li (1983) which is based on a quadratic form of \eqref{stat-std}.
\end{example}

\begin{example}[Assessment of the shock persistency, or the long memory effects in the stationary volatility]
\label{ex2}
In order to determine the strength of the shock persistency, it is common to examine the power autocorrelation function plots. Also, practitioners may consider some model allowing for long memory to capture observed nonlinear effects at large horizons, if significant $\hat{\rho}_0^{(\delta)} (h)$ are observed for $h\gg0$. For instance, the reader is referred to Francq and Zako\"{\i}an (2019), equation (2.53) on the long memory property for the squares correlations of LARCH models. Conversely, if the $\hat{\rho}_0^{(\delta)} (h)$ quickly decrease below adequate confidence bounds, then a classical GARCH model may be sufficient to describe the stochastic volatility.
\end{example}

Usually, no serial correlation for $|\widetilde r_t|^\delta$ means that the data is not subject to stylized financial effects, so $\sigma_t$ is constant. In such a case, the corresponding asset would not be exposed to extreme crises risks. Let us point out that several authors have underlined that a purely deterministic non-constant $\sigma_t$, may produce a spurious significant empirical serial correlation of powers returns. See, e.g., St\u{a}ric\u{a} and Granger (2005) or Patilea and Ra\"{\i}ssi (2014). In order to encompass structural breaks in the variance, a deterministic $\sigma_t$ will be considered in our analysis.

Let us underline that in many cases, $\delta$ is taken to be equal to 2.
However, other values for $\delta$ could be useful in practice, as for instance $\delta=1$, to provide evidence of a possible Taylor effect; see Taylor (2007).

\begin{remark}
In Section \ref{tv-prob-sec} below, $P(a_t=1)$ is allowed to vary over time, which entails that $E(|r_t|^{\delta})$ is also non-constant. Despite this, throughout the paper, we will say that the data displays unconditional heteroscedasticity (or time-varying variance), only when    $\sigma_t^2$ is non-constant. This is because the heteroscedasticity of $(r_t)$ is produced by a time-varying $P(a_t=1)$, which is in some sense spurious, while $\sigma_t^2$ determines the unconditional heteroscedasticity/homoscedasticity behavior of the underlying $(\widetilde r_t)$.

In the same way, denoting by   $\mathcal{F}_{t-1}=\sigma\{\widetilde r_{l}, a_l:l<t\}$, the $\sigma$-field generated by the past values of   $\widetilde  r_t$ and $a_t$,  $E(\widetilde r_t^2|\mathcal{F}_{t-1})$ will be referred to as the volatility. Indeed, if   $E(\widetilde r_t^2|\mathcal{F}_{t-1})$ is stochastic,
then the so-called clustering of extreme values can be observed. Conversely, $E(r_t^{2}|\mathcal{F}_{t-1})$ could be time-varying because of a change in the degree of illiquidity,   while the volatility   $E(\widetilde r_t^2|\mathcal{F}_{t-1})$ is constant (degenerate), which means that the so-called volatility clustering is not present.
\end{remark}

\subsection{Positive and constant zero return probability}

In the following, we focus on the case where only $r_t = a_t \widetilde r_t $ is available, with $\widetilde r_t$ defined as in \eqref{returns}, and   $0<P(a_t=1)<1$. More precisely, we investigate the behavior of the  statistic \eqref{stat-std} which becomes
 the empirical power returns serial correlations
\begin{equation}\label{stat-std}
\widehat{\Gamma}_s^{(\delta)} (m):=\left(\hat{\rho}_s^{(\delta)} (1),\dots,\hat{\rho}_s^{(\delta)} (m)\right)',\:\mbox{with}\:
\hat{\rho}_s^{(\delta)} (h):=\hat{\gamma}_s^{(\delta)} (h)\hat{\gamma}_s^{(\delta)} (0)^{-1},
\end{equation}
where
$$
\hat{\gamma}_s^{(\delta)} (h)=n^{-1}\sum_{t=1+h}^{n}
\left(| r_t|^\delta-\bar{ r}^{(\delta)}\right)\left(| r_{t-h}|^\delta-\bar{ r}^{(\delta)} \right)\quad\text{ and } \quad \bar{ r}^{(\delta)} =n^{-1}\sum_{t=1}^{n}|r_t|^\delta.
$$

Throughout the paper, we assume that the following conditions hold true.

\begin{assumption}[Independence of the innovations and zero-return indicators]\label{ident0}
The realizations $\ldots (\eta_{t-1 },a_{t-1}),(\eta_t,a_t),(\eta_{t+1},a_{t+1})\ldots\in\mathbb R \times \{0,1\}$ are independent.
\end{assumption}

\begin{assumption}[Identification]\label{ident}
The innovations satisfy the conditions
$E(\eta_t|a_t=1)=0$ and $E(\eta_t^{2}|a_t=1)=1$.
\end{assumption}

A common specification in the literature, is to assume that the processes $(a_t)$ and $(\widetilde r_t)$ are independent of each other, with $(\widetilde r_t)$ iid, but allowing for some dependence for $(a_t)$. See, e.g., Dunsmuir and Robinson (1981), Stoffer and Toloi (1992). Nevertheless, such a framework is not appropriate to our task. Indeed, let us underline that our aim is to provide tools for the power correlations analysis, based on a constant $\sigma_t$ and $(\eta_t)$ independent (\emph{i.e.}, no short run dynamics or long memory effects). Moreover, recall that in our setting $r_t=a_t\widetilde r_t$, the $(a_t)$ sequence may be intuitively originated from $\widetilde r_t$. For instance, $a_t=0$ could be the consequence of a low value for $|\widetilde r_t|$. As a consequence, it is hard to believe that $(a_t)$ is dependent, while $(\widetilde r_t)$ is independent. In order to illustrate, this would correspond to observe clusterings of zeros, but no volatility clustering for $r_t\neq0$, which makes no sense in general. However, in order to ensure that relevant situations are within the scope of the paper, we clearly need to allow for dependent $a_t$ and $\eta_t$,    although each of  $(a_t)$ and $(\eta_t)$ is a sequence of  independent random variables. All the above arguments make Assumption \ref{ident0} reasonable for our task. Regarding Assumption \ref{ident}, it is natural to impose an identification condition on $r_t\neq0$ in our framework.
It is worthwhile to notice that, since the latent returns are not observed when $a_t=0$, one could set any values for $E(\eta_t|a_t=0)$  and $E(\eta_t^{2}|a_t=0)$. The data will not allow us to test the choice of the values for these conditional expectations when $a_t=0$. A similar  issue occurs with the so-called Missing-At-Random (MAR) condition in missing data problems, an assumption which is an inherently untestable; see Hristache and Patilea (2017).

The following proposition gives the asymptotic behavior of the powers returns serial correlation when $P(a_t=1)$ is constant. In the following, when $n\rightarrow \infty$, the almost sure convergence is denoted by $\stackrel{a.s.}{\longrightarrow}$, the convergence in distribution is signified by $\stackrel{d}{\longrightarrow}$, and the convergence in probability is denoted by $\stackrel{p}{\longrightarrow}$. Moreover, for any integer $m\geq 1$, $I_m$ is the identity matrix of dimension $m$.

\begin{proposition}\label{propostu}
 Let $\delta >0$, and suppose that $\sigma_t>0$ is constant, with Assumptions \ref{ident0} and  \ref{ident} hold true.  Moreover, assume that all the $(\eta_t,a_t)$'s have the same distribution such that $E\left(|\eta_{t}|^{4\delta}\right)<\infty$ with $0<P(a_t=1)<1$. Then,  for any integer $m\geq 1$, $\sqrt{n}\widehat{\Gamma}_s^{(\delta)} (m)\stackrel{d}{\longrightarrow}\mathcal{N}(0,I_m)$.
\end{proposition}

The reader is referred to McLeod and Li (1983) for the case $\delta=2$ and $P(a_t=1)=1$. In the case $0<P(a_t=1)<1$, using Assumptions \ref{ident0} and  \ref{ident}, for any $h>0$, we have
\begin{multline}
E\left( r_t r_{t-h} \right) = E\left( a_t a_{t-h} \espcond{\sigma_t \sigma_{t-h} \eta_t \eta_{t-h}}{ a_t, a_{t-h} }  \right) \\
\hspace{-4cm} =  \sigma_t^2 \espcond{ \eta_t \eta_{t-h}}{ a_t a_{t-h}=1 }P^2(a_t=1) \\
\hspace{1cm} = \sigma_t^2 \espcond{ \eta_{t-h}E(\eta_t |\eta_{t-h}, a_t a_{t-h}=1) }{ a_t a_{t-h}=1 }P^2(a_t=1) \\
= \sigma_t^2 \espcond{ \eta_{t-h}E(\eta_t |a_t =1) }{ a_t a_{t-h}=1 }P^2(a_t=1) =0,
\end{multline}
and this guarantees that $\widehat{\Gamma}_s^{(\delta)} (m)$ has a centered limit in distribution.
Next, the proof of Proposition \ref{propostu} relies on classical arguments and hence will be omitted. Proposition \ref{propostu} suggests that the standard statistic \eqref{stat-std}, can be used to build confidence intervals for investigating volatility effects when $P(a_t=1)$ is constant. Moreover, consider the test statistic $\mathcal{S}_s^m=n\widehat{\Gamma}_s^{(\delta)} (m)'\widehat{\Gamma}_s^{(\delta)} (m)$. The classical $S_s^m$ test consists in rejecting $H_0^{(\delta)}:\Gamma_0^{(\delta)} (m)=0$ at a given level $\alpha$, if $\mathcal{S}_s^m>\chi_{m,\alpha}^2$, where $\chi_{m,\alpha}^2$ is the $(1-\alpha)$th quantile of the $\chi_{m}^2$ distribution. Hence, in view of the outputs of this section, the usual test can be used safely in presence of a constant non-zero returns probability (i.e. possibly for the Lipigas stock in our real data  example). In the same way the classical autocorrelograms, plotting $\hat{\rho}_s^{(\delta)} (h)$ for a large $h$, can be used to investigate the shocks persistency or long memory effects. For these reasons, in \eqref{stat-std} the subscript "s" stands for "stationary" probability structure. 


\subsection{Time-varying zero return probability}
\label{tv-prob-sec}

In this part, we allow for a time-varying $P(a_t=1)$ as specified in Assumption \ref{tv-prob} below. 

\begin{assumption}[Independent innovations]\label{cst-mom}
The conditional distribution of  $\eta_t$ given  $a_t=1$ does not depend on $t$.
\end{assumption}

\begin{assumption}[Non-constant non-zero returns probability]\label{tv-prob} The time-varying probabilities $P(a_t=1)$ are given by $g(t/n)$,
    where $g(\cdot)$, is a measurable non constant deter\-ministic function, such that $0< g(u)< 1$ on the interval $(0,1]$, and $g(\cdot)$ satisfies a piecewise Lipschitz condition on  $(0,1]$.\footnote{For $u\leq 0$, the function considered in Assumption  \ref{tv-prob} is set constant and equal to  $g(0)=\lim_{u\downarrow0}g(u)$. Throughout the paper, the piecewise Lipschitz condition means: there exists a positive integer $p$ and some mutually disjoint intervals $I_1,\dots,I_p$ with $I_1\cup\dots\cup I_p=(0,1]$ such that $g(u)=\sum_{l=1}^p g_l(u){\bf 1}_{\{u\in I_l\}},$ $u\in(0,1],$ where $g_1(\cdot),\dots,g_p(\cdot)$ are Lipschitz smooth functions on $I_1,\dots,I_p,$ respectively.}
\end{assumption}

From Assumption \ref{ident0}, \ref{cst-mom} and \ref{tv-prob}, the observed process $(r_t)$ is independent, but  not identically distributed as $P(a_t=1)$ is time-varying. This setting will allow us to develop tools for investigating powers dynamics that are robust to a time-varying zero return probability. If the short run higher-order dynamics are investigated, the lack of dynamics in $(a_t)$ is reasonable under $H_0^{(\delta)}$. In the lack of long memory effects, and considering autocorrelations of large orders $h$, then the independence assumption for $(a_t)$ could be at least viewed as approximately valid.
Note that Assumption \ref{tv-prob} makes $0< P(a_t=1)<1$ for all $t\in(0,1]$. This is imposed for writing convenience, and could be relaxed by imposing $0<g(u)<1$ only on a sub-interval of $[0,1]$. The case $P(a_t=1)=0$ for all $t$ does not make sense, while in the  case $P(a_t=1)=1$ for all $t$ there is nothing to add since, with probability 1, $\widetilde r_t$ is always observed. Finally note that the rescaling device, introduced by Dahlhaus (1997), is commonly used for describing long-run trends or structural breaks of times series. See e.g. Phillips and Xu (2006), Xu and Phillips (2008), Cavaliere  and Taylor (2007, 2008), Kew and Harris (2009), Patilea and Ra\"{\i}ssi (2012, 2013) or Wang, Zhao and Li (2019).
The double subscript is skipped in order to keep the notations simple. The following proposition underlines the inadequacy of the statistic \eqref{stat-std} in the non-standard framework defined by Assumption \ref{tv-prob}.

\begin{proposition}\label{propostu2}
Let $\delta >0$, and suppose that $\sigma_t>0$ is constant. Moreover, suppose that Assumptions \ref{ident0} to \ref{tv-prob} hold true with $E\left(|\eta_{t}|^{2\delta}\right)<\infty$. Then, for any integer $m\geq 1$, $\widehat{\Gamma}_s^{(\delta)}(m)\stackrel{a.s.}{\longrightarrow}C_{0,a}\in \mathbb R^m,$ and all the components of the vector $C_{0,a} $ are equal and strictly positive.
\end{proposition}

In view of Proposition \ref{propostu2}, in general we have $S_s^m=O_p(n)$ when $P(a_t=1)$ is time-varying, although $(r_t)$ is independent with constant $\sigma_t$ (\emph{i.e.}, no short-run or long memory volatility effects). Indeed, from \eqref{unique}, the unique case where $C_{0,a} =0$ is when $\int_{0}^{1}g^2(s)ds=\left(\int_{0}^{1}g(s)ds\right)^2$ (\emph{i.e.}, when $g(\cdot)$ is constant almost everywhere, which is not allowed by Assumption \ref{tv-prob}). As a consequence, the usual test for short-run volatility dynamics, could spuriously reject $H_0^{(\delta)}$. More specifically, the $S_s^m$ test is not able to distinguish between the time-varying probability of observing a daily price change, and the volatility clustering. Also, examining the autocorrelogram for large $h$, the $\hat{\rho}_s^{(\delta)} (h)$ will tend to spuriously suggest long memory effects for $\sigma_t$ (e.g. IGARCH or LARCH effects).

\subsection{Time-varying zero return probability and variance}\label{sec:siga}

As pointed out in the Introduction, several facts can be behind a non-constant $P(a_t=1)$. Intuitively, it is sometimes reasonable to think that such probability evolutions, may be accompanied by structural changes in the unconditional variance. For instance, the quick development of emerging markets can lead to a gradual increase of both the liquidity and the unconditional variance of companies returns. Also, a specific event in a given company returns history, may affect the unconditional variance together with its liquidity. Hence, the following assumption is considered.

\begin{assumption}[Time-varying deterministic variance]\label{tv-variance} The $\sigma_t$'s are given by the relationship $\sigma_t = v(t/n)$ where
$v(\cdot)$ is a deterministic function on the interval $(0,1]$, satisfying a piecewise Lipschitz condition on  $(0,1]$, and such that $0< \inf_{r\in(0,1]}v(r)\leq \sup_{r\in(0,1]}v(r)<\infty$.
\end{assumption}

When $P(a_t=1)=1$, Mikosch and St\u{a}ric\u{a} (2004) and Patilea and Ra\"{\i}ssi (2014) underlined that a non-constant $\sigma_t$ as in Assumption \ref{tv-variance}, may spuriously generate significant standard empirical powers correlations (even for large lags $h$). The following result extends the above finding to the case $0<P(a_t=1)<1$.

\begin{proposition}\label{propostu2v}
Let $\delta >0$ and suppose that the Assumptions \ref{ident0} to \ref{tv-variance} hold true with $E\left(|\eta_{t}|^{2\delta}\right)<\infty$.
Then,  for any integer $m\geq 1$, $\widehat{\Gamma}_s^{(\delta)}(m)\stackrel{a.s.}{\longrightarrow}C_{0,\sigma}\in \mathbb R^m$. If $v^\delta (\cdot)g(\cdot)$ is a constant function, then $C_{0,\sigma}$ is the null vector, otherwise all the components of the vector $C_{0,\sigma} $ are equal and strictly positive.
\end{proposition}

Note that a constant $0<P(a_t=1)<1$ is excluded from Proposition \ref{propostu2v}, as $g(\cdot)$ is time-varying in Assumption \ref{tv-prob}. In the sake of conciseness, a result for such a situation is not stated, as it leads to a similar output as in Proposition \ref{propostu2v}. In the framework of Proposition \ref{propostu2v}, the classical autocorrelations are doubly misleading for the powers correlations analysis. Firstly, they are unable to distinguish between a non-constant daily trade probability and a volatility clustering, as above. Secondly, it emerges that they are unable to carry out correctly the task they are originally intended for: do not suggest stochastic effects while only structural variance changes are observed.

\section{Corrected statistics}\label{sec:corr}

We now provide robust tools, taking into account non-constant  $P(a_t=1)$ and time-varying volatility. For the moment, we consider empirical serial correlations of powers returns under the unrealistic assumptions that $P(a_t=1)$ is given when $\sigma_t$ is constant, and that  $P(a_t=1)$ and $\sigma_t$ are given in the general case.

\subsection{Time-varying zero return probability case}\label{sec:corr1}

First, we consider empirical serial correlations of powers returns corrected for the time-varying probability $P(a_t=1)$. For any integer $m\geq 1$, consider
\begin{equation}\label{stat-n-std-1}
\widehat{\Gamma}_{ns}^{(\delta)} (m) := \left(\hat{\rho}_{ns}^{(\delta)} (1),\dots,\hat{\rho}_{ns}^{(\delta)} (m)\right)',\:\mbox{ with }\:
\hat{\rho}_{ns}^{(\delta)} (h) := \hat{\gamma}_{ns}^{(\delta)} (h)\hat{\gamma}_{ns}^{(\delta)} (0)^{-1},
\end{equation}
where
$$
\hat{\gamma}_{ns}^{(\delta)} (h)=n^{-1}\sum_{t=1+h}^{n}
\left(|r_t|^\delta-\bar{r}^{(\delta)}\frac{P(a_t=1)}{\bar{a}}\right)\left(|r_{t-h}|^\delta-\bar{r}^{(\delta)}\frac{P(a_{t-h}=1)}{\bar{a}}\right),
$$
and   $\bar{a}=n^{-1}\sum_{t=1}^{n}P(a_t=1) $. The subscript "ns" stands for "non stationary probability". The following proposition gives the asymptotic behavior of $\widehat{\Gamma}_{ns}^{(\delta)}(m)$ when volatility effects are not present.

\begin{proposition}\label{propostu3}
Assume that $\sigma_t>0$ is constant. Suppose that Assumptions \ref{ident0}-\ref{tv-prob} hold true with $E\left(|\eta_{t}|^{4\delta}\right)<\infty$. Then,  for any integer $m\geq 1$, we have
$\sqrt{n}\widehat{\Gamma}_{ns}^{(\delta)}(m)\stackrel{d}{\longrightarrow}\mathcal{N}(0,\varsigma I_m),$ as $n\to\infty$,
where $\varsigma=\tilde{\varsigma}/\dot{\varsigma}$, with $\tilde{\varsigma}$ and $\dot{\varsigma}$ are given in \eqref{sigtilde} and \eqref{sigcheck}.
\end{proposition}

From the expressions \eqref{sigtilde} and \eqref{sigcheck}  in the Appendix, we can see that the standard asymptotic normal distribution is retrieved when the non-zero returns probability is constant. However, as opposed to the classical statistic, the correction proposed here does not diverge to infinity, in the presence of a time-varying non-zero returns probability. Proposition \ref{propostu3} indicates  that $\widehat{\Gamma}_{ns}^{(\delta)}(m)$ could be used to test
$H_0^{(\delta)}:\Gamma_0^{(\delta)} (m)=0$. However, in order to build feasible tools for inference, we need to estimate $P(a_t=1)$ pointwise. In Section \ref{sec:corr_adapt}, we show how such estimators and the associated feasible tools could be built.

Let us now investigate two situations where a test of
the hypothesis $H_0^{(\delta)}:\Gamma_0^{(\delta)} (m)=0$ based on a quadratic form of the corrected autocorrelations \eqref{stat-n-std-1}, would reject the null hypothesis. Firstly, we consider the case where $\sigma_t$ is stochastic, but with constant $E(|r_t|^\delta|a_t=1)$, i.e. the unconditional variance is constant. To this aim, an AR(1) structure is considered for ease of demonstration (see Proposition \ref{propostu4} below). More precisely, let

\begin{equation}\label{error_vol}
\varepsilon_t=\frac{\eta_t}{E(|\eta_t|^\delta)^{1/\delta}},
\quad\mbox{and}\quad
\tilde{\sigma}_t=\sigma_t\times\left(E(|\eta_t|^\delta)^{1/\delta}\right),
\end{equation}
so that $(\varepsilon_t)$ is iid, $E(\varepsilon_t)=0$, $E(|\varepsilon_t|^\delta)=1$. Note that $\varepsilon_t$ and $\tilde{\sigma}_t$ are just obtained by rescaling $\eta_t$ and $\sigma_t$. We suppose that $E(|\varepsilon_t|^{2\delta})<\infty$. 
The stochastic volatility $\tilde{\sigma}_t$ follows a power ARCH(1) model
\begin{equation}\label{vol_vol}
\tilde{\sigma}_t^\delta=c+\theta|\tilde{\sigma}_{t-1}\varepsilon_{t-1}|^\delta,
\end{equation}
with $0<\theta<1$ and $c>0$ (see Francq and Zako\"{\i}an (2019), Definition 4.5).
Recall that $\mathcal{F}_{t-1}=\sigma\{r_{l}:l<t\}$ is the $\sigma$-field generated by the past values of $r_t$. Since we have $E(|\varepsilon_t|^\delta)=1$, note that from \eqref{vol_vol} we may write $|\widetilde{r}_t|^\delta=c+\theta|\widetilde{ r}_{t-1}|^\delta+e_t$, where we recall that $r_t=a_t\widetilde r_t$, and $(e_t)$ is such that $E(e_t|\mathcal{F}_{t-1})=0$.

\begin{proposition}\label{propostu4}
Suppose that Assumption \ref{tv-prob} and \eqref{vol_vol}
hold true. Assume that $(a_t)$ and $(\widetilde{r}_t)$ are independent. Then, for any integer $m\geq 1$,  $\widehat{\Gamma}_{ns}^{(\delta)}(m)\stackrel{p}{\longrightarrow}C_1,$ where $C_1=(\theta,\theta^2,\dots,\theta^m)'$.
\end{proposition}

Proposition \ref{propostu4} shows that a test based on $\widehat{\Gamma}_{ns}^{(\delta)}(m)$ would reject
$H_0^{(\delta)}:\Gamma_0^{(\delta)} (m)=0$ when \eqref{vol_vol} happens with $\theta>0$. This is exactly what one could expect. Note that we make an independence assumption between $(a_t)$ and $(\widetilde{r}_t)$ to simplify the proof of Proposition \ref{propostu4}. In contrast to this case, we consider a situation which would lead to the spurious rejection of the null hypothesis $H_0^{(\delta)}$ when the test is based on $\widehat{\Gamma}_{ns}^{(\delta)}(m)$. This situation is that of  a deterministic specification of the long-run volatility effects. See e.g. Mikosch and St\u{a}ric\u{a} (2004) for this kind of specification. See also Section \ref{sec:siga} above.

\begin{proposition}\label{propostu5}
Suppose that Assumption \ref{ident0}-\ref{tv-prob} hold true with $E(|\eta_t|^{2\delta})<\infty$. Assume that $(\widetilde r_t)$ fulfills \eqref{returns}, where the $\sigma_{t}$'s are given by $\sigma_{t}=v(t/n)$ from Assumption \ref{tv-variance}.
Then, for any integer $m\geq 1$, $\widehat{\Gamma}_{ns}^{(\delta)} (m)\stackrel{a.s.}{\longrightarrow}C_{1,\sigma},$ where $C_{1,\sigma}$ is a vector with all the components  equal and strictly positive, except when both $P(a_t=1)$ and $\sigma_t$ are constant, in which case the components of $C_{1,\sigma}$ are all equal to zero.
\end{proposition}

In the next section, we present a more general correction which could be applied to avoid the wrong conclusion that the latent powers returns of $(\widetilde r_t)$ are autocorrelated when either $P(a_t=1)$, or deterministic $\sigma_t$, or both are time-varying.


\subsection{Time-varying zero return probability and variance}\label{sec:corr2}

In the presence of a time-varying variance, Patilea and Ra\"{\i}ssi (2014) proposed a modified portmanteau test for second-order dynamics. 
In the following, we underline that such an approach may be extended to the case where both  $P(a_t=1)$ and $\sigma_t$ could vary over time. 
Consider
\begin{equation}\label{stat-n-std-tau}
\widehat{\Gamma}_{ns,\sigma}^{(\delta)} (m) := \left(\hat{\rho}_{ns,\sigma}^{(\delta)} (1),\dots,\hat{\rho}_{ns,\sigma}^{(\delta)} (m)\right)',\:\mbox{ with }\:
\hat{\rho}_{ns,\sigma}^{(\delta)} (h) := \hat{\gamma}_{ns,\sigma}^{(\delta)} (h)\hat{\gamma}_{ns,\sigma}^{(\delta)} (0)^{-1},
\end{equation}
where
$$
\hat{\gamma}_{ns,\sigma}^{(\delta)} (h)=n^{-1}\sum_{t=1+h}^{n}
\left\{|r_t|^\delta-E\left(|r_t|^\delta\right)   \right\}\left\{|r_{t-h}|^\delta- E\left(|r_t|^\delta\right) \right\}.
$$
It easy to see that $\widehat{\Gamma}_{ns,\sigma}^{(\delta)} (\cdot) = \widehat{\Gamma}_{ns}^{(\delta)} (\cdot)$ when $\sigma_t$ is constant.

\begin{proposition}\label{propostu4v}
 Suppose that Assumption \ref{ident0}-\ref{tv-variance} hold true with $E\left(|\eta_{t}|^{4\delta}\right)<\infty$. Then, for any integer $m\geq 1$,  we have
$\sqrt{n}\widehat{\Gamma}_{ns,\sigma}^{(\delta)}(m)\stackrel{d}{\longrightarrow}\mathcal{N}(0,\zeta I_m),$
where $\zeta=\tilde{\zeta}/\dot{\zeta}$, with $\tilde{\zeta}$ and $\dot{\zeta}$ are given in \eqref{zetanum} and \eqref{zetaden}.
\end{proposition}

\begin{proposition}\label{propostu4-bis}
Suppose that Assumption \ref{tv-prob} holds true. Suppose that $\widetilde r_t = \tilde{\sigma}_t \varepsilon_t$ where $\varepsilon_t$ and $\tilde{\sigma}_t$ are defined as in  \eqref{error_vol} and \eqref{vol_vol}. Suppose that $(a_t)$ and $(\widetilde{r}_t)$ are independent. Then, for any integer $m\geq 1$,  $\widehat{\Gamma}_{ns,\sigma}^{(\delta)}(m)\stackrel{p}{\longrightarrow}C_1,$ where $C_1=(\theta,\theta^2,\dots,\theta^m)'$.
\end{proposition}

Proposition \ref{propostu4v} extends Proposition 3 of Patilea and Ra\"{\i}ssi (2014), in the case of illiquid assets. In view of Proposition \ref{propostu4v}, $\widehat{\Gamma}_{ns,\sigma}^{(\delta)}(m)$ is robust with respect to both time-varying unconditional variance and zero returns probability. Proposition \ref{propostu4-bis} illustrates the ability of $\widehat{\Gamma}_{ns,\sigma}^{(\delta)}(m)$, to detect stochastic effects in the volatility. Nevertheless, when the variance of $\widetilde r_t$ is constant, it can be expected that the $\widehat{\Gamma}_{ns,\sigma}^{(\delta)}(m)$, may suffer from a loss of power when compared with the more specific $\widehat{\Gamma}_{ns}^{(\delta)}(m)$. Indeed, the volatility clustering associated with a serial correlation of powers returns, could in part be confused with a time-varying $E\left(|r_t|^\delta\right)$.

\begin{remark}[Conditional mean filtering]
Sometimes practitioners investigate the serial correlation of $(r_t)$, prior to studying the higher order dynamics of the data. If $corr(r_t,r_{t-h})=0$ is rejected for some $h>0$, then some model is fitted to handle the conditional mean of $(r_t)$. In a second step, the power correlations are assessed, but using the residuals obtained from the conditional mean estimation.  
Nevertheless, note that in many cases, the classical Box test for investigating $corr(r_t,r_{t-h})$, $h>0$ is used, which could spuriously suggest the existence of a conditional mean when the data exhibit non-standard behaviors. Indeed, Romano and Thombs (1996) underlined the inadequacy of the classical Box test in the presence of stationary nonlinear patterns (\emph{i.e.} the so-called GARCH effect) in the observed series. Also, Patilea and Ra\"{\i}ssi (2013) showed that the classical Box test is oversized when the variance is time-varying as in Assumption \ref{tv-variance}. In the same way, Ra\"{\i}ssi (2020) found that the classical Box test tends to reject the no serial correlation for $(r_t)$ erroneously when the zero returns probability is non-constant as in Assumption \ref{tv-prob}. For these reasons, we believe that the power correlation study can be carried out, using the $(r_t)$ directly in most of the cases.

Nevertheless, if someone wants to proceed with a conditional mean, a AR($k$) model of the form

\begin{equation*}
r_t=\vartheta_{00}+\vartheta_{01}r_{t-1}+\dots+\vartheta_{0k}r_{t-k}+u_t,
\end{equation*}
may be adjusted to the returns. The errors $u_t$ are uncorrelated, and $\vartheta_0=(\vartheta_{00},\vartheta_{01},\dots,$ $\vartheta_{0k})'$ is a vector of parameters.  Introducing such a step in a few cases, may be motivated from some evidence from the data (\emph{e.g.} market inefficiency). We define the OLS estimator $\hat{\vartheta}$. In the framework of Assumption \ref{tv-variance}, it is shown that $\sqrt{n}\left(\hat{\vartheta}-\vartheta_0\right)=O_p(1)$ (see Xu and Phillips (2008)). The power empirical correlations are computed using the residuals $\hat{u}_t$ instead of $u_t$. However, taking $\delta=2$, and using similar arguments to those of McLeod and Li (1983), equation (2.3), it can be shown that the conditional mean estimation step can be asymptotically neglected, i.e. the $(u_t)$ may be assumed observed for establishing asymptotic results. 
\end{remark}

\section{Data-driven tools for power series correlation analysis}\label{sec:corr_adapt}

Let us now propose adaptive, data-driven  versions of the statistics \eqref{stat-n-std-1} and \eqref{stat-n-std-tau}. These statistics involve quantities which are usually unknown, that are $P(a_t=1)$ and $\sigma_t$, and have to be estimated from the data. Here, we propose nonparametric kernel estimates of these quantities, which are then plugged into the expressions  \eqref{stat-n-std-1} and \eqref{stat-n-std-tau} in the place of the theoretical counterparts. This results in feasible, adaptive tools for power series correlation analysis. For simplicity, only the case of the statistic \eqref{stat-n-std-tau} will be considered, that in  \eqref{stat-n-std-1} being just a special case of it.

For a bandwidth $b$ and a kernel function $K(\cdot)$,  let
$$w_{tj}(b)= (nb)^{-1} K\left((t-j)/(nb)\right).$$
The Nadaraya-Watson estimators  for $P(a_t=1)$ and $E(|r_t|^{\delta})$, with a fixed regular design on $(0,1]$, are
\begin{equation}\label{Sig_NP}
\widehat{P(a_t=1)}=\sum_{j=1}^{n} w_{tj}(b_a) a_j,
\end{equation}
and
\begin{equation}\label{Sig_NP2}
 \widehat{E(|r_t|^{\delta})} =\sum_{j=1}^{n} w_{tj}(b_\tau) |r_j|^{\delta},
\end{equation}
respectively. The bandwidths $b_a$ and $b_\tau$ could be selected by usual  procedures for kernel regression implemented in the existing software.

With at hand, the nonparametric estimators, we are now in position to introduce the following powers correlations, which are robust to changes in the zero returns probability: for any integer $m\geq 1$,
\begin{equation}\label{stat-n-std-1b}
\widetilde{\Gamma}_{ns}^{(\delta)} (m) := \left(\tilde{\rho}_{ns}^{(\delta)} (1),\dots,\tilde{\rho}_{ns}^{(\delta)} (m)\right)',\:\mbox{ with }\:
\tilde{\rho}_{ns}^{(\delta)} (h) := \tilde{\gamma}_{ns}^{(\delta)} (h)\tilde{\gamma}_{ns}^{(\delta)} (0)^{-1},
\end{equation}
with
$$
\tilde{\gamma}_{ns}^{(\delta)} (h)=n^{-1}\sum_{t=1+h}^{n}
\left(|r_t|^\delta-\bar{r}^{(\delta)}\frac{\widehat{ P(a_t=1) }}{\bar{\hat{a}}}\right)\left(|r_{t-h}|^\delta-\bar{r}^{(\delta)}\frac{\widehat{ P(a_{t-h}=1) } }{\bar{\hat{a}}}\right),
$$
and   $\bar{\hat{a}}=n^{-1}\sum_{t=1}^{n}\widehat{P(a_t=1) }$.
Moreover,  the following powers correlations, are robust to changes in both the zero returns probability and the variance: for any integer $m\geq 1$,
\begin{equation}\label{stat-n-std}
\widetilde{\Gamma}_{ns,\sigma}^{(\delta)} (m):=\left(\tilde{\rho}_{ns,\sigma}^{(\delta)} (1),\dots,\tilde{\rho}_{ns,\sigma}^{(\delta)} (m)\right)',\:\mbox{with}\:
\tilde{\rho}_{ns,\sigma}^{(\delta)} (h):=\tilde{\gamma}_{ns,\sigma}^{(\delta)} (h)\tilde{\gamma}_{ns,\sigma}^{(\delta)} (0)^{-1},
\end{equation}
where
\begin{equation}\label{ert1}
\tilde{\gamma}_{ns,\sigma}^{(\delta)}(h) = n^{-1}\sum_{t=1+h}^{n}
\left( |r_t|^\delta - \widehat{E(|r_t|^{\delta} } \right)
\left(|r_{t-h}|^\delta - \widehat{E(|r_{t-h}| ^{\delta} }  \right).
\end{equation}
In the following, the estimators  \eqref{stat-n-std-1b} and \eqref{stat-n-std}  will be called the adaptive estimators of the powers autocorrelations.

We show that, for any fixed $m\geq 1$, the differences
\begin{equation}\label{as_equivb}
\widetilde{\Gamma}_{ns}^{(\delta)} (m)-\widehat{\Gamma}_{ns}^{(\delta)} (m)\quad \text{ and } \quad \widetilde{\Gamma}_{ns,\sigma}^{(\delta)} (m)-\widehat{\Gamma}_{ns,\sigma}^{(\delta)} (m),
\end{equation}
are asymptotically negligible when the null hypothesis $H_0^{(\delta)}:\Gamma_0^{(\delta)} (m)=0$ holds true. Thus $\widetilde{\Gamma}_{ns}^{(\delta)} (m)$  and  $\widetilde{\Gamma}_{ns,\sigma}^{(\delta)} (m)$ could replace $\widehat{\Gamma}_{ns}^{(\delta)} (m)$ and $\widehat{\Gamma}_{ns,\sigma}^{(\delta)} (m)$, respectively,  for testing $H_0^{(\delta)}$. The following assumption is used for the asymptotic equivalence between the unfeasible and the adaptive statistics.

\begin{assumption}[Kernel smoothing assumption]\label{kern-sm}

\begin{itemize}
\item[(i)] The kernel $K(\cdot)$ is a bounded  symmetric density function defined on the real line such that $K(\cdot)$ is non-decreasing on $(-\infty, 0]$ and decreasing on $[0,\infty)$ and $\int_\mathbb{R} |v|K(v)dv < \infty$. The function $K(\cdot)$ is differentiable except for a finite number of points and the derivative $K^\prime(\cdot)$  is a bounded  integrable function.
Moreover, the Fourier Transform $\mathcal{F}[K](\cdot)$ of $K(\cdot)$ satisfies $\int_{\mathbb{R}}  \left| s \mathcal{F}[K](s) \right|ds <\infty$.

\item[(ii)] The bandwidths $b_a$ and $b_\tau$ are taken in the range $\mathcal{B}_n = [c_{min} b_n, c_{max} b_n]$ with $0< c_{min}< c_{max}< \infty$ and  $nb_n^{4} + 1/nb_n^{2+\gamma} \rightarrow 0$  as $n\rightarrow \infty$, for some $\gamma >0$.
\end{itemize}
\end{assumption}

The technical conditions imposed on the kernel function are mild and satisfied by the kernels commonly used. Following the approach of Patilea and Ra\"{\i}ssi (2012), the bandwidths belong to a range $\mathcal{B}_n $ defined by some constants $c_{min}, c_{max}>0$  and a deterministic sequence $b_n$, $n\geq 1$, decreasing to zero at a suitable rate. For simplicity, we consider a common range for $b_a$ and $b_\sigma$, but the extension to the case where the two bandwidths have different decrease rates is straightforward. The asymptotic results are derived uniformly with respect to the bandwidths in such ranges. This provides theoretical grounds for the practitioner who  selects the bandwidth using the sample.

\begin{proposition}\label{propostu5v}
The Assumption \ref{kern-sm} is satisfied. Let  $m\geq 1$.
Under  the conditions of Proposition \ref{propostu3},
\begin{equation}\label{uty1}
\sqrt{n}\left| \widehat{\Gamma}_{ns}^{(\delta)}(m) - \widetilde{\Gamma}_{ns}^{(\delta)}(m) \right|  \stackrel{p}{\longrightarrow} 0,
\end{equation}
uniformly with respect to $b_a\in\mathcal B_n$, and under  the conditions of Proposition \ref{propostu4v}
\begin{equation}\label{uty2}
\sqrt{n}  \left| \widehat{\Gamma}_{ns,\sigma}^{(\delta)}(m) - \widetilde{\Gamma}_{ns,\sigma}^{(\delta)}(m) \right|  \stackrel{p}{\longrightarrow} 0,
\end{equation}
uniformly with respect to $b_\tau\in\mathcal B_n$.
\end{proposition}

As a direct consequence of Proposition \ref{propostu5v}, for any choice of the bandwidth   $b_\tau\in\mathcal B_n$,
$\sqrt{n}\widetilde{\Gamma}_{ns,\sigma}^{(\delta)}(m)\stackrel{d}{\longrightarrow}\mathcal{N}(0,\zeta I_m),$
where $\zeta=\tilde{\zeta}/\dot{\zeta}$, with $\tilde{\zeta}$ and $\dot{\zeta}$  given in \eqref{zetanum} and \eqref{zetaden}. A similar statement holds true for $\sqrt{n}\widetilde{\Gamma}_{ns}^{(\delta)}(m)$ with $\tilde{\zeta}$ and $\dot{\zeta}$ given in  \eqref{sigtilde} and \eqref{sigcheck}.

Concerning the properties of the adaptive tools,  let us notice that  when the null hypothesis
$H_0^{(\delta)}:\Gamma_0^{(\delta)} (m)=0$ is violated,  $\widetilde{\Gamma}_{ns,\sigma}^{(\delta)}(m)$ is able to detect stochastic effects in the volatility. This aspect is formalized in the following result.

\begin{proposition}\label{propostu5-bis}
Assume that the conditions of Proposition \ref{propostu4-bis}, and Assumption \ref{kern-sm} hold true. Then the conclusion of Proposition \ref{propostu4-bis} remains valid with $\widetilde{\Gamma}_{ns,\sigma}^{(\delta)}(m)$ in place of $\widehat{\Gamma}_{ns,\sigma}^{(\delta)}(m)$.
\end{proposition}

\section{Numerical illustrations}\label{num.exp}

As mentioned throughout the paper, we test $H_0^{(\delta)}\:\mbox{vs.}\:H_1^{(\delta)}$ for a relatively small horizon, for instance $m=5$ (see Example \ref{ex1}). On the other hand, the power autocorrelations are compared with suitable confidence bands for large $h$ (see Example \ref{ex2}). Note that these kind of outputs correspond to the usual practice in time series econometrics.

Using tools obtained directly from the asymptotic results of the above sections lead to size distortions. Indeed, estimating $P(a_t=1)$ or  $E\left(|r_t|^\delta\right)$ could be sometimes difficult, in particular when abrupt shifts are present, as in the Provida stock example. Then, following Patilea and Ra\"{\i}ssi (2014), bootstrap tests and confidence bounds will be built, using: 
\begin{equation}\label{ert2b}
\widehat{\gamma}_{*ns}^{(\delta)}(h)=n^{-1}\sum_{t=1+h}^{n}
\xi_t \left(|r_t|^\delta- \bar{r}^{(\delta)}\frac{\widehat{P(a_t=1)}}{\hat{\bar{a}}}\right) \xi_{t-h} \left(|r_{t-h}|^\delta- \bar{r}^{(\delta)}\frac{\widehat{P(a_t=1)}}{\hat{\bar{a}}} \right),
\end{equation}
 where $\hat{\bar{a}}:=n^{-1}\sum_{t=1}^{n}\widehat{P(a_t=1)}$, and
\begin{equation}\label{ert2c}
\widehat{\gamma}_{*ns,\sigma}^{(\delta)}(h)=n^{-1}\sum_{t=1+h}^{n}
\xi_t \left(|r_t|^\delta- \widehat{E(|r_t| ^{\delta})}\right) \xi_{t-h} \left(|r_{t-h}|^\delta- \widehat{E(|r_{t-h}| ^{\delta})} \right).
\end{equation}
The bootstrap disturbances process $(\xi_t)$ is independent from $(r_t)$, and such that $E(\xi_t)=0$ and $E(\xi_t^2)=1$. Several choices for $(\xi_t)$ are available in the literature. For the sake of conciseness, we only display the outputs for the Mammen distribution:
$$
  P\left(\xi_t=-0.5(\sqrt{5}-1)\right) = \frac{0.5(\sqrt{5}+1)}{\sqrt{5}},\quad
  P\left(\xi_t=0.5(\sqrt{5}+1)\right) = \frac{0.5(\sqrt{5}-1)}{\sqrt{5}}.
$$
Considering the Rademacher distribution for $(\xi_t)$ leads to the same general conclusions (the outputs are not displayed here).

In all the experiments, $B=3999$ bootstrap replicates are generated. For testing $H_0^{(\delta)}\:\mbox{vs.}\:H_1^{(\delta)}$, we define the statistic robust to a time-varying probability only (denoted by "RP" hereafter) $\widetilde{\mathcal{S}}_{ns}^m=n\widetilde{\Gamma}_{ns}^\delta(m)'\widetilde{\Gamma}_{ns}^\delta(m)$. Then, we can compute bootstrap counterparts $\widetilde{\mathcal{S}}_{*ns}^{m}$ using the $\widehat{\gamma}_{*ns}^{(\delta)}(h)$'s. The RP bootstrap test for short-run higher order dynamics consists in rejecting $H_0^{(\delta)}$, at level $\alpha$ if $\widetilde{\mathcal{S}}_{ns}^m$ is greater than the $(1-\alpha)$-quantile of the $\widetilde{\mathcal{S}}_{*ns}^{m}$'s. Also, the bootstrap $\alpha$ level confidence intervals for the RP autocorrelations may be obtained using the $\alpha/2$ and $1-\alpha/2$ quantiles of the bootstrap powers autocorrelations replicates $\widetilde{\Gamma}_{*ns}^\delta(m)$ obtained from the $\widehat{\gamma}_{*ns}^{(\delta)}(h)$'s. Similarly, tools that are robust to both time-varying probability and variance (denoted "RPV" hereafter), may be computed from the $\widehat{\gamma}_{*ns,\sigma}^{(\delta)}(h)$'s. In the following, we refer to the "RP" and "RPV" tools, as the "adaptive" tools. We focus on the case $\delta=1$ (\emph{i.e.}, the absolute values autocorrelations analysis), the extension to other values of $\delta$ is straightforward.

For implementing the RP and RPV powers autocorrelations,  the normal kernel is used. The bandwidth $b_a$ minimizes the following  leave-one-out cross validation criterion (LOOCV) for the estimation of the probability structure:

\begin{equation}\label{cv-px}
CV(b_a)=\sum_{t=1}^{n}\left(\widehat{P_{-t}(a_t=1)}-a_t\right)^2,
\end{equation}
where $\widehat{P_{-t}(a_t=1)}$ is the estimator of $P(a_t=1)$ obtained by omitting $a_t$. Similarly, for the RPV powers autocorrelations, for any $\delta\geq 1$,  the bandwidth $b_\tau$ 
can be obtained by minimizing the following LOOCV criterion:
\begin{equation}\label{cv-vx}
CV(b_\tau)= \sum_{t=1}^{n} \left(\widehat{E_{-t}(|r_t|^\delta)}-|r_t|^\delta\right)^2,
\end{equation}
with obvious notation.

\subsection{Monte Carlo experiments}

Several kinds of data are simulated to assess the finite sample properties of the adaptive and classical tools. Recall that in practice $\widetilde{r}_t=\sigma_t\eta_t$ is partially observed, while $r_t=a_t\widetilde{r}_t$ is fully observed.
The $\sigma_t$, $\eta_t$, and the $P(a_t=1)$ are specified according to several situations of interest:


\begin{itemize}
\item[(a)] (Constant volatility) Neither short-run nor long-run dynamics are considered for the volatility, by setting $\eta_t\sim\mathcal{N}(0,1)$ independent and $\sigma_t=1$ for all $t$.
\item[(b)]  (Short-run dynamics in the volatility, without any long-run effects) The
following classical GARCH(1,1) model is used:
\begin{eqnarray}
  \widetilde{r}_t&=& \sigma_t\eta_t \label{garch-mod}\\
  \sigma_t&=&0.01+0.1\widetilde{r}_{t-1}^2+0.8\sigma_{t-1}^2, \nonumber
\end{eqnarray}
where $\eta_t\sim\mathcal{N}(0,1)$. Clearly, $(\widetilde{r}_t)$ is strictly stationary, but conditionally heteroscedastic, such that $E(\widetilde{r}_t)=0$ and $E(\widetilde{r}_t^2)<\infty$. In this case, we set $a_t=0$ if $\widetilde{r}_t<\mu$, where $\mu$ is the median of $\widetilde{r}_t$.
\item[(c)] (Long-run variance effects) The deterministic specification as in Mikosch and St\u{a}ric\u{a} (2004) is used:
 $\eta_t\sim\mathcal{N}(0,1)$ independent, with $\sigma_t=v(t/n)$, and $v(r)=1_{\{(0,0.4]\}}(r)+(5r-1)\times1_{\{(0.4,0.6]\}}(r)+2\times1_{\{(0.6,1]\}}(r)$ (see Proposition \ref{propostu5}). The latter case describes a quick shift in the unconditional variance.
\end{itemize}

In order to investigate the effects of the zero returns on the power autocorrelations analysis, we set
$a_t=\dot{a}_t\ddot{a}_t$, with $\ddot{a}_t=0$ if $|\eta_t|\leq0.063$, with $\eta_t\sim\mathcal{N}(0,1)$, so that $P(\ddot{a}_t=1)=0.9$, and
examine the following cases for the long-run structure of the zero returns probability:

\begin{itemize}
\item[(1)] (Constant zero returns probability) $P(\dot{a}_t=1)$ is taken constant equal to 0.5. 
\item[(2)] (Time-varying probability) A quick shift is depicted by $P(\dot{a}_t=1)=\dot{g}(t/n)$, with $\dot{g}(r)=0.2\times1_{\{(0,0.4]\}}(r)+(3.5r-1.2)\times1_{\{(0.4,0.6]\}}(r)+0.9\times1_{\{(0.6,1]\}}(r)$. Among other possible events, such a behavior can be encountered when a company is subject to a take-over bid (decreasing quick shift as in the Provida stock case), or when a capital increase is achieved (increasing quick shift as in the Molymet stock case). Here, we restrict ourselves to an increasing quick shift to avoid lengthy outputs. Also, progressive changes could commonly be observed in the zero returns probability structure, as for the Cruzados or Conchatoro stocks. However, for the sake of conciseness, the outputs for such a situation are not displayed, as they are similar to the non-constant case presented here.
\end{itemize}


In all our experiments, $R=5000$ independent trajectories of length $n=100,200,$ $400,800$ are used.
The outputs are presented in Figure \ref{fig-index} and Table \ref{conf-bound-homo}-\ref{port-reject}. In Table \ref{port-reject} and Figure \ref{fig-index}, the rejection frequencies of the portmanteau tests are reported. In Table \ref{conf-bound-homo}-\ref{conf-bound-shift}, the frequencies of the adaptive and the classical autocorrelations outside their 95\% confidence bounds are given.\\

We begin to investigate the situations with no long-run or short-run effects in the volatility, i.e. (a)-(1) and (a)-(2). In Table \ref{conf-bound-homo}, the outputs for a constant $P(a_t=1)$ are given (the (a)-(1) case). It can be seen that the frequencies are close to the 5\% level in general. This can be explained by the fact that the classical and adaptive tools are all asymptotically valid when $P(a_t=1)$ is constant (see Proposition \ref{propostu}). The same can be stated for the portmanteau tests from Table \ref{port-reject}. We only note some slight finite sample distortions for the RP and RPV tools when compared to the classical tools. Also it seems that the RP tools have some slight advantages compared with the RPV tools. Nonetheless, the good asymptotic results for the classical autocorrelations, are no longer available in presence of a time-varying zero returns probability (see the case (a)-(2) in Table \ref{conf-bound-shift1}). It can be observed that the frequencies of classical autocorrelations outside its confidence bounds, increase as a large sample size is taken, although there is no stochastic dynamics in the volatility. Conversely, the adaptive autocorrelations have satisfactory outputs. We only note some small sample distortions. In the same way, it emerges from the middle panel of Table \ref{port-reject}, that the standard portmanteau test is inadequate in the presence of a non-constant zero returns probability. In contrast, the adaptive portmanteau tests seem to well control the type I error. Finally, let us underline that some finite sample advantage can be noticed for the more specific RP tools, when compared to the RPV tools. These simulations results are inline with Proposition \ref{propostu2}, \ref{propostu3}, \ref{propostu4v}, \ref{propostu5v} and \ref{propostu5-bis}.

Next, the properties of the power autocorrelations in the presence of a deterministic but non-constant unconditional variance is analyzed. For simplicity, we restrict ourselves to the case where the probability is also time-varying (i.e. the (c)-(2) case). Indeed, the outputs for a constant probability, and a time-varying variance, lead to the same conclusions. The outputs in Table \ref{conf-bound-shift} show that the classical and RP autocorrelations are not able to distinguish between stochastic long run dynamics (for instance IGARCH effects), and structural changes in the variance structure. This is expected, as they are not intended to take into account unconditional heteroscedastic patterns. On the other hand, it can be noted that the RPV autocorrelations display acceptable results. The same comments can be made from the bottom panel of Table \ref{port-reject}, where only the RPV portmanteau test reasonably control the type I error.

The ability of the autocorrelations to detect second order dynamics is now studied. To this aim we consider case (b). Note that the probability and variance are set constant in order to make a fair comparison between the standard and adaptive tools. Indeed, all the methods are valid in this particular case. To avoid lengthy outputs, we only display the rejection frequencies of the portmanteau tests in Figure \ref{fig-index}. Some loss of power can be observed for the adaptive portmanteau tests, when compared to the standard portmanteau test. Nevertheless, the adaptive tests have shown a clear ability to detect higher order dynamics. Finally, let us highlight that the RPV test suffers from a loss of power when compared to the simpler RP test.\\

As a conclusion, 
if the time-varying non-zero returns probability is accompanied by some structural unconditional variance changes, the tools based on the RPV autocorrelations should be used. Nevertheless, the more specific RP and classical autocorrelations achieve better size and power results in their respective fields of application. Hence, in view of our simulation results, it is important to adequately determine the analysis framework for illiquid assets. The set-up to be considered can be identified by examining the stock's past events, or graphical representations. 
Furthermore, descriptive statistics estimating $\int g$ and $E(|r_t|^\delta|a_t=1)$ over time can be used. 


\subsection{Real data analysis}

In this part, the serial correlations of the absolute returns ($\delta=1$) presented in the Introduction are investigated. To this aim, we first try to identify the appropriate autocorrelations to be used for the analysis of each stock, considering the evidence from the data. Recall that the returns are plotted in Figure \ref{ns-stocks}, together with the kernel estimator of the daily zero return probability. Recall also that underlying information strengthening the idea of a structural change is available for some stocks. In order to complement this piece of information, we display the changes in the zero returns probability and in the variance employing descriptive tools inspired from the variance profile of Cavaliere and Taylor (2007) in Figure \ref{ns-stocks-profiles}. More precisely, the probability profile is given by $\hat{p}(s):=\left(\sum_{t=1}^{n}a_t\right)^{-1}\left(\sum_{t=1}^{[ns]}a_t\right)$, $0<s\leq 1$. For the absolute return profile, let $r_{t_1},r_{t_2},\dots,r_{t_\nu}$, $t_\nu\leq n$, the non-zero values of the sequence $(r_t)$. The absolute return profile is then defined by $\widehat{ar}(s):=\left(\sum_{j=1}^{\nu}|r_{t_j}|\right)^{-1}\left(\sum_{j=1}^{[\nu s]}|r_{t_j}|\right)$. In the following, we present the arguments leading us to focus on a particular kind of autocorrelation for each stock.

\begin{itemize}
\item Provida: A take-over bid occurred in 2013, which caused a rapid decrease of the non zero returns probability. This observation is confirmed by Figure \ref{ns-stocks} and \ref{ns-stocks-profiles}. On the other hand, Figure \ref{ns-stocks-profiles} suggests that the unconditional variance is constant for this stock. It then seems that the RP autocorrelations can be safely considered.
\item Molymet: The company achieved an increase of capital in 2010. In Figure \ref{ns-stocks}, this results in a sharp increase of the non-zero returns probability. In addition, from Figure \ref{ns-stocks-profiles}, we can strongly suspect that the 2010 increase of capital also caused a structural break in the unconditional variance structure. In view of the above arguments, we rely on the RPV autocorrelations.
\item Conchatoro: From Figure \ref{ns-stocks} and \ref{ns-stocks-profiles}, a progressive increase of the non-zero returns probability can be seen throughout most of the decade 2000-2010. This can be explained by the developing Chilean stocks market during this period. Figure \ref{ns-stocks-profiles} suggests that the unconditional variance is constant, so the RP autocorrelations can be used.
\item Cruzados: From Figure \ref{ns-stocks} and \ref{ns-stocks-profiles}, a long-range decrease of the non-zero returns probability structure can be strongly suspected. On the other hand, Figure \ref{ns-stocks-profiles} show that the unconditional variance seems constant, so it is advisable to use the RP autocorrelations.
\item Lipigas: Figure \ref{ns-stocks} and \ref{ns-stocks-profiles} strongly show that the non-zero returns probability is constant. From Figure  \ref{ns-stocks-profiles}, the variance seems also constant. For this case, the classical autocorrelations could be examined.
\item Las Condes: Figure \ref{ns-stocks} and \ref{ns-stocks-profiles} clearly lead us to conclude that both the non-zero returns probability and unconditional variance are constant. For these reasons, the classical autocorrelations can be certainly considered for this stock.
\end{itemize}

In Table \ref{stat-test}, the relevant p-values of the portmanteau tests are in bold underlined typescript. The sample sizes of the full series are also given in Table \ref{stat-test}.
In Figure \ref{fig-index-2}-\ref{fig-index-3}, the classical and adaptive autocorrelations are also plotted 
to assess the significance of the autocorrelations at different horizons. For short-run dynamics, $1\leq h\leq 10$ is studied, complementing the outputs of the portmanteau tests. In order to investigate some possible strong persistency of the volatility shocks, $10<h\leq20$ are examined. Long-run effects are studied considering $20<h\leq60$. Recall that the adaptive tools are applied, generating $B=3999$ bootstrap replicates.

%

From Figure \ref{fig-index-2}-\ref{fig-index-3} and Table \ref{stat-test}, short-run serial correlations for the absolute returns can be clearly noted for all the stocks, if the relevant autocorrelations are examined. The volatility clustering that can be observed in Figure \ref{ns-stocks} for all the stocks, confirms our diagnostic. Likewise, focusing on large $h$ in Figure \ref{fig-index-2}-\ref{fig-index-3}, lead us to conclude that there is no strong persistency of the shocks, unless for the Conchatoro company. On the other hand, let us underline that if the effect of a non-stationary zero returns probability is neglected, the classical autocorrelations might suggest the spurious existence of IGARCH effects in the returns (see the autocorrelations plots of the Provida, Molymet and Cruzados stocks). Finally, it is likely that the RPV tools may suffer from a lack of power in detecting higher correlations. This can be seen from the outputs of the Cruzados and Lipigas stocks in Table \ref{stat-test}. The classical portmanteau (resp. RP) test seems more reliable for the Lipigas (resp. Cruzados) stock, in view of the volatility clustering in Figure \ref{ns-stocks}.

\section{Conclusion}
\label{concl.sec}

It emerges from our study that the classical autocorrelations for higher dynamics, are still valid when the returns sequence of an illiquid asset is stationary. Nevertheless, a time-varying zero return probability can be commonly encountered in practice. For instance, illiquid stocks of emerging markets, or stocks with specific events often exhibit such a behavior. In addition, it should be noted that these kinds of facts may also generate unconditional heteroscedasticity. As a consequence, adaptive tools are proposed to take into account these non-standard situations. The properties of the different powers correlations studied here, bring to light the following main facts for the stocks displaying a time-varying degree of illiquidity.



\begin{itemize}
\item It is commonly admitted that short-run volatility dynamics are present for most of the stocks. Nevertheless, our outputs clearly indicate that the classical autocorrelations generally deliver doubtful conclusions in our framework. As a consequence, it is important to have available tools to ensure that such effects are not spurious. The analysis of the adaptive tools for the stocks studied in this paper, seems to confirm rigorously that short-run dynamics is a common pattern for illiquid assets.
\item Long-run effects are frequently assumed for the stock returns modeling. 
    However, if the classical autocorrelations are used, such long-run volatility effects are likely to be spurious in view of the non-constant behavior of the zero return probability. A practitioner could at least have a misleading assessment of a strong shocks persistency for this kind of stocks. These erroneous assessments are the consequence of the confusion between the volatility clustering, and the inhomogeneous non zero returns distribution over time. Also, considering the classical autocorrelations can lead to inaccurate conclusions on the consequences of some events on the stock returns (as for instance a take-over bid or a capital increase spuriously producing long memory effects). In contrast, the more accurate analyses using the adaptive tools, suggest that the shock persistency may be often low for the non-stationary illiquid stocks. 
\end{itemize}


\section*{References}
\begin{description}
\item[]{\sc Andrews, D.W.K.} (1988) Laws of large numbers for dependent non-identically distributed random variables. \textit{Econometric Theory} 4, 458--67.
\item[]{\sc Cavaliere, G., and Taylor, A.M.R.} (2007) Time-transformed unit-root tests for models with non-stationary volatility. \textit{Journal of Time Series Analysis} 29, 300--330.
\item[]{\sc Cavaliere, G., and Taylor, A.M.R.} (2008) Bootstrap unit root tests for time Series with nonstationary volatility. \textit{Econometric Theory} 24, 43--71.
\item[]{\sc Dahlhaus, R.} (1997) Fitting time series models to nonstationary processes. \textit{Annals of Statistics} 25, 1--37.
\item[]{\sc Dunsmuir, W., and  Robinson, P.M.} (1981) Estimation of time
series models in the presence of missing data.  \emph{Journal of the American Statistical Association} 76, 560--568.
\item[]{\sc Engle, R.F.} (1982) Autoregressive conditional heteroscedasticity with estimates of the variance of United Kingdom inflation. \textit{Econometrica} 50, 987--1007.
\item[]{\sc Engle, R.F., and Rangel, J.G.} (2008) The spline GARCH model for unconditional volatility and its global macroeconomic causes. \textit{Review of Financial Studies} 21, 1187--1222.
\item[]{\sc Francq, C., and Zako\"{i}an, J-M.} (2019) \textit{GARCH models : structure, statistical inference, and financial applications}. Wiley.
\item[]{\sc Fry\'{z}lewicz, P.} (2005) Modelling and forecasting financial log-returns as locally stationary wavelet processes. \textit{Journal of Applied Statistics} 32, 503--528.
\item[]{\sc  Hristache, M., and Patilea, V.} (2017) Conditional moment models with data missing at random. \emph{Biometrika} 104, 735--742.
\item[]{\sc Kew, H., and Harris, D.} (2009) Heteroskedasticity-robust testing for a fractional unit root. \textit{Econometric Theory} 25, 1734--1753.
\item[]{\sc Lesmond, D.A.} (2005) Liquidity of emerging markets. \textit{Journal of Financial Economics} 77, 411--452.
\item[]{\sc McLeod, A.I., and Li, W. K.} (1983) Diagnostic checking ARMA time series models using squared-residual autocorrelations. \textit{Journal Time Series Analysis} 4, 269--273.
\item[]{\sc Mikosch, T., and St\u{a}ric\u{a}, C.} (2004) Nonstationarities in financial time series, the long-range dependence, and the IGARCH effects. \textit{Review of Economics and Statistics} 86, 378--390.
\item[]{\sc Parzen, E.} (1963) On spectral analysis with missing observations and amplitude modulation. \textit{Sankhya Series A} 25, 383--392.
\item[]{\sc Patilea, V., and Ra\"{i}ssi, H.} (2012) Adaptive estimation of vector autoregressive models with time-varying variance: application to testing linear causality in mean. \textit{Journal of Statistical Planning and Inference} 142, 2891--2912.
\item[]{\sc Patilea, V., and Ra\"{i}ssi, H.} (2013) Corrected portmanteau tests for VAR models with time-varying variance. \textit{Journal of Multivariate Analysis} 116, 190--207.
\item[]{\sc Patilea, V., and Ra\"{i}ssi, H.} (2014) Testing second order dynamics for autoregressive processes in presence of time-varying variance. \emph{Journal of the American Statistical Association} 109, 1099--1111.
\item[]{\sc  Phillips, P.C.B., and Xu, K.L.} (2006) Inference in Autoregression under Heteroskedasticity. \textit{Journal of Time Series Analysis} 27, 289--308.
\item[]{\sc Ra\"{i}ssi, H.} (2018) Testing normality for unconditionally heteroscedastic macroeconomic variables. \textit{Economic Modelling} 70, 140--146.
\item[]{\sc Ra\"{i}ssi, H.} (2020) On the correlation analysis of illiquid stocks.  \texttt{arXiv:2008.06168v1}.
\item[]{\sc Romano, J.P., and Thombs, L. A.} (1996) Inference for autocorrelations under weak assumptions. \textit{Journal of the American Statistical Association} 91, 590--600.
\item[]{\sc Sen, P. K., and Singer, J. M.} (1993) \textit{Large Sample Methods In Statistics}. Chapman \& Hall.
\item[]{\sc St\u{a}ric\u{a}, C., and Granger, C.} (2005) Nonstationarities in stock returns. \textit{Review of Economics and Statistics} 87, 503--522.
\item[]{\sc Sherman, R.P.} (1994) Maximal Inequalities for Degenerate $U-$Processes with Applications to Optimization Estimators. \textit{The Annals of Statistics}  22, 439--459.
\item[]{\sc Stoffer, D.S., and Toloi, C.M.C.} (1992) A note on the Ljung-Box-Pierce portmanteau statistic with missing data. \textit{Statistics and Probability Letters} 13, 391--396.
\item[]{\sc Taylor, S.J.} (2007) \textit{Modelling Financial Time Series} (2nd edition). World Scientific.
\item[]{\sc Wang, S., Zhao, Q., and Li, Y.} Testing for no-cointegration under time-varying variance. \textit{Economics Letters} 182, 45--49.
\item[]{\sc Xu, K.L., and Phillips, P.C.B.} (2008) Adaptive estimation of autoregressive models with time-varying variances. \textit{Journal of Econometrics} 142, 265--280.

\end{description}

\newpage

\section*{Proofs}

\begin{proof}[Proof of Proposition \ref{propostu2}]
First note that under the conditions of Proposition \ref{propostu2},   $(r_t)$ is a sequence of independent variables and the conditional distribution of $r_t$ given  $a_t=1$ does not depend on $t$. We begin with the numerator of $\hat{\rho}_s^{(\delta)}(h)$. From the Kolmogorov SLLN for independent but non-identically random variables (see, for instance,  Sen and Singer (1993), Theorem 2.3.10), we have

\begin{eqnarray*}
\bar{r}^{(\delta)} =n^{-1}\sum_{t=1}^{n}|r_t|^\delta&\stackrel{a.s.}{\longrightarrow}&n^{-1}\sum_{t=1}^{n}E\left(|r_t|^\delta\right)\\
&=&E(|r_t|^\delta|a_t=1)\left\{n^{-1}\sum_{t=1}^{n}P(a_t=1)\right\}.
\end{eqnarray*}
For the equality, we used Assumption \ref{cst-mom}, and the fact that $\sigma_t$ is constant which imply that $E(|r_t|^\delta|a_t=1)$ does not depend on $t$.
Using the usual convergence property of Riemann sums, and in view of the Lipschitz continuous condition and the finite number of breaks conditions in Assumption \ref{tv-prob}, we write
$$n^{-1}\sum_{t=1}^{n}P(a_t=1)=\int_{0}^{1}g(s)ds+O(n^{-1}).$$
Then deduce that
\begin{equation}\label{bara}
\bar{r}^{(\delta)} \stackrel{a.s.}{\longrightarrow} E(|r_t|^\delta|a_t=1)\int_{0}^{1}g(s)ds.
\end{equation}
By similar arguments,
$$
n^{-1}\sum_{t=1}^{n}|r_t|^{2\delta} \stackrel{a.s.}{\longrightarrow}E(|r_t|^{2\delta}|a_t=1)\int_{0}^{1}g(s)ds.
$$
Next, let $1\leq h \leq m$ and consider the sequence of independent terms
$$\mathcal S = (\ldots,|r_{t-h-1} r_{t-2h-1}|^\delta, |r_t r_{t-h}|^\delta, |r_{t+h+1}r_{t+1}|^\delta,\ldots).$$
We can split the sequence $(|r_tr_{t-h}|^\delta)$ into $h+1$ sub-sequences of independent terms which are $\mathcal S$, and the $h$ sequences obtained by successive applications of the shift operator to $\mathcal S$. For each of these $h+1$ sub-sequences, we can apply the Kolmogorov SLLN and the arguments from above, and deduce their almost sure convergence. Our assumptions imply that the $h+1$ limits are equal, and thus we  obtain
$$
n^{-1}\sum_{t=1+h}^{n}|r_t|^\delta|r_{t-h}|^\delta \stackrel{a.s.}{\longrightarrow} E^2(|r_t|^\delta|a_t=1)\int_{0}^{1}g^2(s)ds.
$$
Gathering the above facts, and using some elementary computations we have
\begin{equation}\label{unique}
n^{-1}\!\sum_{t=1+h}^{n}\!\! \left(\!|r_t|^\delta\! - \bar{r}^{(\delta)}   \right)\!\! \left(\!|r_{t-h}|^\delta\! -\bar{r}^{(\delta)} \right)\!
\stackrel{a.s.}{\longrightarrow}\!E^2(|r_t|^\delta|a_t=1)\! \left[\int_{0}^{1}\!\!g^2(s)ds-\left(\int_{0}^{1}\!\! g(s)ds\right)^2\right]\!.
\end{equation}
By Cauchy-Schwarz inequality, the limit in the last display is positive and could be zero only if $g(\cdot)$ is almost everywhere constant.
For the denominator of $\hat{\rho}^{(\delta)}(h)$, similarly  we obtain
$$n^{-1}\sum_{t=1}^{n}\!\left(\!|r_t|^\delta\! -\bar{r}^{(\delta)} \right)^2\stackrel{a.s.}{\longrightarrow}
\left[\!E(|r_t|^{2\delta}|a_t=1)\!\int_{0}^{1}\!\!g(s)ds-\left(E(|r_t|^{\delta}|a_t=1)\!\int_{0}^{1}\!g(s)ds\right)^2\right],
$$
and the limit is strictly positive as soon as the $r_t$'s are not all equal. The stated result follows.
\end{proof}

\quad

\begin{proof}[Proof of Proposition \ref{propostu2v}]
The same arguments as in the the proof of Proposition \ref{propostu2} apply.
Here we can write
\begin{equation}\label{bara2}
 \bar{r}^{(\delta)} \stackrel{a.s.}{\longrightarrow} E(|\eta_t|^\delta|a_t=1)\int_{0}^{1}v^\delta (s) g(s)ds,
\end{equation}
and
$$
n^{-1}\sum_{t=1+h}^{n}|r_t|^\delta|r_{t-h}|^\delta \stackrel{a.s.}{\longrightarrow} E^2(|\eta_t|^\delta|a_t=1)\int_{0}^{1}v^{2\delta } (s)g^2(s)ds.
$$
We deduce
\begin{multline}\label{unique2}
n^{-1}\sum_{t=1+h}^{n}\!\! \left(|r_t|^\delta -  \bar{r}^{(\delta)}   \right) \left(|r_{t-h}|^\delta -  \bar{r}^{(\delta)}   \right)\\
\stackrel{a.s.}{\longrightarrow}E^2(|r_t|^\delta|a_t=1) \left[\int_{0}^{1}v^{2\delta }  (s) g^2(s)ds-\left(\int_{0}^{1}v ^{\delta } (s) g(s)ds\right)^2\right]\geq 0 .
\end{multline}
Moreover,
$$
n^{-1}\sum_{t=1}^{n}|r_t|^{2\delta} \stackrel{a.s.}{\longrightarrow}E(|\eta_t|^{2\delta}|a_t=1)\int_{0}^{1}v^ {2\delta }  (s) g(s)ds.
$$
and thus
\begin{multline}\label{unique3}
n^{-1}\sum_{t=1}^{n}\left(|r_t|^\delta - \bar{r}^{(\delta)}  \right)^2\\
\stackrel{a.s.}{\longrightarrow}
\left[E(|\eta_t|^{2\delta}|a_t=1)\int_{0}^{1} v^{2\delta }  (s) g(s)ds-\left(E(|\eta_t|^{\delta}|a_t=1)\int_{0}^{1}v^\delta (s) g(s)ds\right)^2\right] >0.
\end{multline}
\end{proof}

\medskip

\begin{proof}[Proof of Proposition \ref{propostu3}]
As above, under the conditions of Proposition \ref{propostu3},   $(r_t)$ is a sequence of independent variables and the conditional distribution of $r_t$ given  $a_t=1$ does not depend on $t$.
Note that, by a simple  error bound for Riemann sums,
\begin{equation}\label{sqe1}
\bar{a} - \int_{0}^{1}g(s)ds
=O(n^{-1}).
\end{equation}
Similarly, by Assumption \ref{cst-mom}, 
\begin{multline}\label{sqe2}
\bar{r}^{(\delta)}  - \sigma^\delta E(|\eta_t|^\delta\mid a_t=1) \int_{0}^{1} g(s)ds   =   \frac{1}{n}\sum_{t=1}^{n}\!\left\{|r_t|^\delta - E(|r_t|^\delta)\right\} \\ + \sigma^\delta E(|\eta_t|^\delta\mid a_t=1) \left[ \frac{1}{n}\sum_{t=1}^{n}  g(t/n) -  \int_{0}^{1} g(s)ds\right]   \\ =  \frac{1}{n}\sum_{t=1}^{n}\left\{|r_t|^\delta - E(|r_t|^\delta)\right\} + O(n^{-1}).
\end{multline}
Moreover, by simple variance calculation,
\begin{equation}\label{sqe2b}
 \frac{1}{n}\sum_{t=1}^{n}\left\{|r_t|^\delta - E(|r_t|^\delta)\right\}= O_p(n^{-1/2}) .
\end{equation}
Since, for any $x_0,y_0$ with $y_0>0$, and any $x,y$ close to $x_0$ and $y_0$ respectively,
$$
\frac{x}{y} = \frac{x_0}{y_0} + \frac{1}{y_0}(x-x_0) - \frac{x_0}{y^2_0}(y-y_0) + o(| x-x_0 | +| y-y_0 |  ) ,
$$
we deduce
\begin{multline}\label{sqe2c}
P(a_t=1)\frac{\bar{r}^{(\delta)} }{\bar{a}} = P(a_t=1)\frac{\sigma^\delta E(|\eta_t|^\delta\mid a_t=1)  \int_{0}^{1}g(s)ds }{\int_{0}^{1}g(s)ds  }+O_p(n^{-1/2}) \\ = E(|r_t|^\delta) +O_p(n^{-1/2}) .
\end{multline}
From this we obtain
\begin{eqnarray*}
n^{\frac{1}{2}}\hat{\gamma}_{ns}^{(\delta)} (h)=n^{-\frac{1}{2}}\sum_{t=1+h}^{n}
\left(|r_t|^\delta-E(|r_t|^\delta)\right)\left(|r_{t-h}|^\delta-E(|r_{t-h}|^\delta)\right)+o_p(1).
\end{eqnarray*}
Next, define the $m$-dimensional vector $Z_t$, with $h$ components $$Z_{ht}=\left(|r_t|^\delta-E(|r_t|^\delta)\right)\left(|r_{t-h}|^\delta-E(|r_{t-h}|^\delta)\right).$$
It is clear that any linear combination $\lambda'Z_t$ fulfills the Lindeberg condition. In addition, $\lambda'Z_t$ is $m$-dependent, so that using the LLN for $L^1$-mixingales of Andrews (1988), we write
$$
n^{-1}\sum_{t=1+h}^{n}E((\lambda'Z_t)^2|\mathcal{F}_{t-1})\stackrel{p}{\longrightarrow}\varpi>0.
$$
Then, from the CLT for   martingale difference sequences (see Theorem A.3 in Francq and Zako\"{\i}an (2019)), and using the Cramer-Wold device, we obtain
\begin{equation}\label{ANproof}
n^{-\frac{1}{2}}\sum_{t=1+h}^{n}Z_t\stackrel{d}{\longrightarrow}\mathcal{N}(0,\tilde{\varsigma}I_m).
\end{equation}
Note that the covariance matrix is diagonal from our independence assumption. The diagonal components correspond to the asymptotic variance of
$$
\tilde{\varsigma}=Var_{as}\left(n^{-\frac{1}{2}}\sum_{t=1+h}^{n}Z_{ht}\right).
$$
By some tedious computations, and using again the independence assumption,  we obtain
\begin{multline}\label{sigtilde}
\tilde{\varsigma}=E\left\{|r_t|^{2\delta}|a_t=1\right\}^2\left(\int_{0}^{1}g^2(s)ds\right)\\
-2E\left\{|r_t|^{2\delta}|a_t=1\right\}E\left\{|r_t|^\delta|a_t=1\right\}^2\left(\int_{0}^{1}g^3(s)ds\right)\\
+E\left\{|r_t|^\delta|a_t=1\right\}^4\left(\int_{0}^{1}g^4(s)ds\right).
\end{multline}

On the other hand, as above we have
$$
n^{-1}\sum_{t=1}^{n}\left(|r_t|^\delta-\frac{\bar{r}_t^{(\delta)}}{\bar{a}}P(a_t=1)\right)^2=
n^{-1}\sum_{t=1}^{n}\left(|r_t|^\delta-E\left(|r_t|^\delta\right)\right)^2+o_p(1).
$$
Using the Kolmogorov SLLN, we obtain
\begin{equation}\label{pproof}
n^{-1}\sum_{t=1}^{n}\left(|r_t|^\delta-E\left(|r_t|^\delta\right)\right)^2\stackrel{a.s.}{\longrightarrow}
n^{-1}\sum_{t=1}^{n}E\left\{\left(|r_t|^\delta-E\left(|r_t|^\delta\right)\right)^2\right\}.
\end{equation}
By easy computations, 
\begin{multline}\label{sigcheck}
\left[n^{-1}\sum_{t=1}^{n}E\left\{\left(|r_t|^\delta-E\left(|r_t|^\delta\right)\right)^2\right\}\right]^2\\
\hspace{-6cm} =E\left\{|r_t|^{2\delta}|a_t=1\right\}^2\left(\int_{0}^{1}g(s)ds\right)^2\\
\hspace{1cm} - 2E\left\{|r_t|^{2\delta}|a_t=1\right\}E\left\{|r_t|^\delta|a_t=1\right\}^2\left(\int_{0}^{1}g^2(s)ds\right)\left(\int_{0}^{1}g(s)ds\right)\\
+E\left\{|r_t|^\delta|a_t=1\right\}^4\left(\int_{0}^{1}g^2(s)ds\right)^2+o(1)
\\  =:\dot{\varsigma}+o(1).
\end{multline}
The desired result follows from \eqref{ANproof}, \eqref{pproof} and the Slutsky Lemma.
\end{proof}

\medskip

\begin{proof}[Proof of Proposition \ref{propostu4}]
As before, we begin with the numerator. Under the assumptions of Proposition \ref{propostu4}, the sequence $|r_t|^\delta$ is a $L^2$-mixingale (see the arguments in Davidson (1994), Example 16.2). Then, using the LLN of Andrews (1988) for $L^1$-mixingales, we write
$$\bar{r}^{(\delta)} =n^{-1}\sum_{t=1}^{n}|r_t|^\delta\stackrel{p}{\longrightarrow}n^{-1}\sum_{t=1}^{n}E\left(|r_t|^\delta\right).$$
From \eqref{error_vol} and \eqref{vol_vol}, it is clear that $E\left(|r_t|^\delta|a_t=1\right)$ is constant. Hence, we obtain
\begin{equation}\label{levadura}
\frac{\bar{r}^{(\delta)}}{\bar{a}}P(a_t=1)\stackrel{p}{\longrightarrow}E\left(|r_t|^\delta\right).
\end{equation}
From \eqref{levadura}, and using again the LLN of Andrews (1988), it is easy to see that
$$\hat{\gamma}_{ns}^\delta(h)\stackrel{p}{\longrightarrow}n^{-1}\sum_{t=1+h}^{n}
E\left\{\left(|r_t|^\delta-E(|r_t|^\delta)\right)\left(|r_{t-h}|^\delta-E(|r_{t-h}|^\delta)\right)\right\}.$$
In view of our Lipschitz condition with a finite number of breaks in Assumption \ref{tv-prob}, the AR structure of $(|r_t|^\delta)$, and the independence of $(a_t)$ and $(\widetilde{r}_t)$, we then have
\begin{equation}\label{blabla1}
\hat{\gamma}_{ns}^{(\delta)}(h)\stackrel{p}{\longrightarrow}n^{-1}\sum_{t=1+h}^{n}
\theta^hE\left\{\left(|r_{t-h}|^\delta-E(|r_{t-h}|^\delta)\right)^2\right\}+O(n^{-1}).
\end{equation}
Some computations give
\begin{equation}\label{blabla2}
\hat{\gamma}_{ns}^{(\delta)} (h)\stackrel{p}{\longrightarrow}
\theta^h\left[E\left(|r_t|^{2\delta}|a_t=1\right)\int_{0}^{1}g(s)ds-E\left(|r_t|^{\delta}|a_t=1\right)^2\int_{0}^{1}g^2(s)ds\right]\geq0.
\end{equation}
For the denominator, similarly to \eqref{blabla1} and \eqref{blabla2}, we write
\begin{equation*}
\hat{\gamma}_{ns}^ {(\delta)} (0)\stackrel{p}{\longrightarrow}
E\left(|r_t|^{2\delta}|a_t=1\right)\int_{0}^{1}g(s)ds-E\left(|r_t|^{\delta}|a_t=1\right)^2\int_{0}^{1}g^2(s)ds\geq0,
\end{equation*}
and hence we obtain the desired result.
\end{proof}

\medskip

\begin{proof}[Proof of Proposition \ref{propostu5}]
From the Kolmogorov SLLN, and using the same rescaling arguments as above, we have
$$
\frac{\bar{r}^{(\delta)}}{\bar{a}}\stackrel{a.s.}{\longrightarrow}
\frac{\int_{0}^{1}v^\delta(s)g(s)ds}{\int_{0}^{1}g(s)ds}E\left(|\eta_t|^\delta|a_t=1\right).
$$
Then, we clearly obtain
\begin{eqnarray*}
\hat{\gamma}_{ns}^{(\delta)} (h)&=&n^{-1}\sum_{t=1+h}^{n}
E\left\{\left(|r_t|^\delta-\frac{\int_{0}^{1}v^\delta(s)g(s)ds}{\int_{0}^{1}g(s)ds}E\left(|\eta_t|^\delta|a_t=1\right)P(a_{t}=1)\right)\right.\\&\times&
\left.\left(|r_{t-h}|^\delta-\frac{\int_{0}^{1}v^\delta(s)g(s)ds}{\int_{0}^{1}g(s)ds}E\left(|\eta_t|^\delta|a_t=1\right)P(a_{t-h}=1)\right)\right\}
+o_p(1).
\end{eqnarray*}
In the above sum, the terms are $h$-dependent. Then, again using the LLN of Andrews (1988), by some tedious computations,
\begin{eqnarray*}
\hat{\gamma}_{ns}^{(\delta)} (h)&\stackrel{p}{\longrightarrow}&E\left(|\eta_t|^\delta|a_t=1\right)^2\left[\left(\int_{0}^{1}v^{2\delta}(s)g^2(s)ds\right)
\right.\\&-&2\left(\int_{0}^{1}v^{\delta}(s)g^2(s)ds\right)\frac{\int_{0}^{1}v^{\delta}(s)g(s)ds}{\int_{0}^{1}g(s)ds}
\\&+&\left.\left(\frac{\int_{0}^{1}v^{\delta}(s)g(s)ds}{\int_{0}^{1}g(s)ds}\right)^2\int_{0}^{1}g^2(s)ds\right].
\end{eqnarray*}
Using the inequality
$$
\int_{0}^{1}v^{2\delta}(s)G_2(s)ds \geq \left(\int_{0}^{1}v^{\delta}(s)G_2(s)ds\right)^2 \quad\text{with}\quad
G_2(s)= \frac{g^2(s)}{\int_{0}^{1}g^2(u)du},
$$
we deduce that the limit $\hat{\gamma}_{ns}^{(\delta)} (h)$ is larger or equal to
$$
E\left(|\eta_t|^\delta|a_t=1\right)^2 \left(\int_{0}^{1}g^2(s)ds\right)^{-2}\\
\times \left\{  \int_{0}^{1}v^{\delta}(s)G_2(s)ds - \int_{0}^{1}v^{\delta}(s)G_1(s)ds\right\} ^2,
$$
where $G_1(s)= g(s)\left[\int_{0}^{1}g(u)du\right]^{-1}$. Thus, it is easy to see that the limit $\hat{\gamma}_{ns}^{(\delta)} (h)$ is non negative, and could  be equal to zero if and only if both functions $g(\cdot)$ and $v(\cdot)$ are constant.

On the other hand, for the denominator, using similar arguments as above we write
\begin{eqnarray*}
\hat{\gamma}_{ns}^{(\delta)}(0) &\stackrel{p}{\longrightarrow}&E\left(|\eta_t|^{2\delta}|a_t=1\right)\int_{0}^{1}v^{2\delta}(s)g(s)ds
\\&+&E\left(|\eta_t|^\delta|a_t=1\right)^2\left[-2\left(\int_{0}^{1}v^{\delta}(s)g^2(s)ds\right)\frac{\int_{0}^{1}v^{\delta}(s)g(s)ds}{\int_{0}^{1}g(s)ds}
\right.\\&+&\left.\left(\frac{\int_{0}^{1}v^{\delta}(s)g(s)ds}{\int_{0}^{1}g(s)ds}\right)^2\int_{0}^{1}g^2(s)ds\right],
\end{eqnarray*}
so the stated result follows.
\end{proof}

\begin{proof}[Proof of Proposition \ref{propostu4v}]

Like in  the proof of Proposition \ref{propostu3},
we define the $m$-dimensional vector $Z_t$, with $h$ components $$Z_{ht}=\left(|r_t|^\delta-E(|r_t|^\delta)\right)\left(|r_{t-h}|^\delta-E(|r_{t-h}|^\delta)\right),\qquad 1\leq h\leq m.$$
Clearly, the convergence \eqref{ANproof} is still guaranteed by the  CLT for martingale difference sequences. The asymptotic covariance is diagonal, using the independence property. Using some algebra we obtain
\begin{eqnarray}
\tilde{\zeta}&=&E\left\{|\eta_t|^{2\delta}|a_t=1\right\}^2\left(\int_{0}^{1}v^{4\delta}(s)g^2(s)ds\right)\nonumber\\
&&-2E\left\{|\eta_t|^{2\delta}|a_t=1\right\}E\left\{|\eta_t|^\delta|a_t=1\right\}^2\left(\int_{0}^{1}v^{4\delta}(s)g^3(s)ds\right)\nonumber\\
&&+E\left\{|\eta_t|^\delta|a_t=1\right\}^4\left(\int_{0}^{1}v^{4\delta}(s)g^4(s)ds\right).\label{zetanum}
\end{eqnarray}
For $\hat{\gamma}_{ns,\sigma}^{(\delta)} (0)$, we note that
\begin{eqnarray}
&&\left[n^{-1}\sum_{t=1}^{n}\left(|r_t|^\delta-E\left(|r_t|^\delta\right)\right)^2\right]^2\nonumber\\
&=&E\left\{|\eta_t|^{2\delta}|a_t=1\right\}^2\left(\int_{0}^{1}v^{2\delta}(s)g(s)ds\right)^2\nonumber\\&-&
2E\left\{|\eta_t|^{2\delta}|a_t=1\right\}E\left\{|\eta_t|^\delta|a_t=1\right\}^2\left(\int_{0}^{1}v^{2\delta}(s)g^2(s)ds\right)\left(\int_{0}^{1}v^{2\delta}(s)g(s)ds\right)\nonumber\\
&+&E\left\{|\eta_t|^\delta|a_t=1\right\}^4\left(\int_{0}^{1}v^{2\delta}(s)g^2(s)ds\right)^2+o_p(1)=:\dot{\zeta}+o_p(1).\label{zetaden}
\end{eqnarray}
Hence, the result follows from the Slutsky Lemma.
\end{proof}

\medskip

\begin{proof}[Proof of Proposition \ref{propostu4-bis}] The proof of Proposition \ref{propostu4-bis} is similar to that of Proposition \ref{propostu4}, and therefore omitted.
\end{proof}

\medskip

\begin{proof}[Proof of Proposition \ref{propostu5v}] We sketch the lines of this technical proof. See also Patilea and Ra\"{i}ssi (2014) for similar arguments and additional technical details. For justifying \eqref{uty2}, it suffices to prove
\begin{equation}\label{equiv_g}
\max_{0\leq h\leq m} \sup_{ 
b_\tau\in\mathcal B_n} n^{1/2} \left| \widetilde{\gamma}_{ns,\sigma}^{(\delta)} (h) -  \widehat{\gamma}_{ns,\sigma}^{(\delta)} (h)\right| = o_p(1).
\end{equation}
We can write
\begin{multline}\label{dsq1}
\widehat{E(|r_{t}| ^{\delta} )}  =  E(|r_{t}| ^{\delta}  )\\+ \sum_{j=1}^{n} w_{tj}(b_\tau) \left\{ E(|r_{j}| ^{\delta})- E(|r_{t}| ^{\delta}  )\right\} + \sum_{j=1}^{n} w_{tj}(b_\tau) \left\{ |r_j|^{\delta} - E(|r_{j}| ^{\delta})\right\}\\ +E(|r_{t}| ^{\delta}  )\left[ \sum_{j=1}^{n} w_{tj}(b_\tau) - 1\right]=: E(|r_{t}| ^{\delta}  ) + B_t+V_t+W_t.
\end{multline}
By our Lipschitz condition for  $v(\cdot)$ and $g(\cdot)$, and the fact that $E\left(|\eta_t|^\delta|a_t=1\right)$ does not depend on $t$,
$$
 \left| E(|r_{j}| ^{\delta})- E(|r_{t}| ^{\delta}  )\right| = E\left(|\eta_t|^\delta|a_t=1\right)\left| v^\delta(j/n)g(j/n) -  v^\delta(t/n)g(t/n)  )\right|\leq C_1 |j-t|/n,
$$
where $C_1$ is some constant. By Lemma 6.1 of Patilea and Ra\"{i}ssi (2014),
$$
|B_t| + |W_t| \leq \frac{c_{max}}{c_{min}} C_1 b_n \left\{1+ \frac{C_2}{nb_n} \right\} + E(|r_{t}| ^{\delta}  )\frac{C_2}{nb_n},
$$
for some constant $C_2$. Note that $B_t$ and $W_t$ are not random. On the other hand, for any $\gamma >0$,  by the results of Sherman (1994),
$V_t = O_p(n^{-1/2} b^{-1-\gamma/2})$ uniformly with respect to $b_\tau\in\mathcal B_n$ and $t$. Thus, replacing
$\widehat{E(|r_{t}| ^{\delta} )} $ by $E(|r_{t}| ^{\delta}  ) + B_t+V_t+W_t$ in the definition of $\tilde{\gamma}_{ns,\sigma}^{(\delta)}(h)$, after some calculations, we obtain \eqref{equiv_g}.

The arguments for justifying  \eqref{uty1} are quite similar and are hence omitted. \end{proof}

\medskip

\begin{proof}[Proof of Proposition \ref{propostu5-bis}]
The justification of this result is a direct consequence of Theorem \ref{propostu5v} and is thus omitted.
\end{proof}

\newpage

\section*{Tables and Figures}

\begin{table}[hh]\!\!\!\!\!\!\!\!\!\!
\begin{center}
\caption{\small{The frequencies (in \%) of adaptive and classical autocorrelations outside their respective nominal 95\% confidence bands, obtained from $R=5000$ independent replications. The absolute value of the returns are considered ($\delta=1$). The constant unconditional variance and probability case (a)-(1), with no second order dynamics. In all the experiments, the confidence intervals for the powers RPV and RP autocorrelations are built using bootstrap replications. Also, the bandwidths for estimating the adaptive powers correlations minimize the CV criteria given in \eqref{cv-px} and \eqref{cv-vx}.}}
\begin{tabular}{|c|c|c|c|c|c|c||c|c|c|}
\cline{2-10}
 \multicolumn{1}{c|}{ }& lags & 1 & 2 & 3 & 4 & 5 & 20 & 40 & 60 \\
\hline
\multirow{4}{*}{\begin{turn}{90}Classical\end{turn}}
& $n=100$ & 4.12 &4.20 &4.20 &4.66 &4.40 &2.30 &1.14 &0.28\\ \cline{2-10}
& $n=200$ & 4.52 &4.46 &4.64 &4.12 &4.76 &3.78 &2.70 &1.68\\ \cline{2-10}
& $n=400$ & 5.36 &4.30 &4.56 &4.90 &4.56 &4.22 &3.36 &2.80\\ \cline{2-10}
& $n=800$ & 5.20 &5.04 &4.70 &4.64 &4.90 &5.02 &4.22 &4.36 \\ \hline\hline
\multirow{4}{*}{\begin{turn}{90}RP\end{turn}}
& $n=100$ & 6.04 &6.16 &5.98 &6.70 &6.48 &5.14 &4.56 &4.08\\ \cline{2-10}
& $n=200$ & 6.06 &5.80 &5.64 &5.62 &6.16 &5.36 &5.48 &4.98\\ \cline{2-10}
& $n=400$ & 5.76 &4.96 &5.82 &5.34 &5.62 &5.08 &5.14 &4.82 \\ \cline{2-10}
& $n=800$ & 5.18 &5.12 &4.98 &4.94 &5.54 &6.08 &5.14 &5.88\\ \hline\hline
\multirow{4}{*}{\begin{turn}{90}RPV\end{turn}}
& $n=100$ &  6.24 &6.28 &6.34 &7.14 &6.76 &5.18 &4.54 &3.84\\ \cline{2-10}
& $n=200$ &  6.04 &6.10 &5.88 &5.60 &6.52 &5.58 &5.82 &5.02\\ \cline{2-10}
& $n=400$ &  6.04 &5.00 &5.78 &5.38 &5.80 &5.20 &5.30 &5.00\\ \cline{2-10}
& $n=800$ &  5.42 &5.24 &5.06 &5.02 &5.58 &6.04 &5.30 &6.00\\ \hline
\end{tabular}
\label{conf-bound-homo}
\end{center}
\end{table}

\begin{table}[hh]\!\!\!\!\!\!\!\!\!\!
\begin{center}
\caption{\small{The frequencies (in \%) of adaptive and classical autocorrelations outside their respective nominal 95\% confidence bands ($\delta=1$), obtained from $R=5000$ independent replications. The daily non zero returns probability has a quick shift from a regime to another, with a constant variance and no second order dynamics (the (a)-(2) case).}}
\begin{tabular}{|c|c|c|c|c|c|c||c|c|c|}
\cline{2-10}
 \multicolumn{1}{c|}{ }& lags & 1 & 2 & 3 & 4 & 5 & 20 & 40 & 60 \\
\hline
\multirow{4}{*}{\begin{turn}{90}Classical\end{turn}}
& $n=100$ &  46.04 &44.44 &42.30 &40.50 &40.68 &8.64 &0.64 &1.38\\ \cline{2-10}
& $n=200$ &  71.68 &70.12 &71.44 &70.04 &69.76 &53.64 &17.60  &1.52\\ \cline{2-10}
& $n=400$ &  94.70 &94.64 &94.74 &94.54 &93.66 &91.90 &83.10 &67.36\\ \cline{2-10}
& $n=800$ &  99.88 &99.80 &99.90 &99.92 &99.88 &99.7 &99.8 &99.3\\ \hline\hline
\multirow{4}{*}{\begin{turn}{90}RP\end{turn}}
& $n=100$ &  6.20 &5.78 &5.62 &6.16 &5.04 &4.14 &3.44 &1.82\\ \cline{2-10}
& $n=200$ &  5.88 &6.24 &5.92 &6.04 &5.64 &4.66 &4.66 &5.40\\ \cline{2-10}
& $n=400$ &  5.70 &5.36 &5.24 &5.28 &5.70 &5.02 &4.88 &4.60\\ \cline{2-10}
& $n=800$ &  5.58 &5.50 &5.26 &5.64 &5.32 &6.06 &5.28 &5.80\\ \hline\hline
\multirow{4}{*}{\begin{turn}{90}RPV\end{turn}}
& $n=100$ &  6.70 &6.36 &7.10 &7.86 &6.34 &4.42 &3.42 &1.78\\ \cline{2-10}
& $n=200$ &  6.70 &6.98 &6.96 &6.60 &6.50 &4.72 &4.52 &5.44\\ \cline{2-10}
& $n=400$ &  6.12 &6.34 &5.92 &6.20 &6.20 &5.26 &5.04 &4.48\\ \cline{2-10}
& $n=800$ &  5.86 &5.60 &5.70 &5.80 &6.10 &6.44 &5.54 &5.68\\ \hline
\end{tabular}
\label{conf-bound-shift1}
\end{center}
\end{table}

\begin{table}[hh]\!\!\!\!\!\!\!\!\!\!
\begin{center}
\caption{\small{The frequencies (in \%) of adaptive and classical autocorrelations outside their respective nominal 95\% confidence bands ($\delta=1$), obtained from $R=5000$ independent replications. The daily non zero returns probability and unconditional variance have a quick shift from a regime to another, with no second order dynamics (the (c)-(2) case).}}
\begin{tabular}{|c|c|c|c|c|c|c||c|c|c|}
\cline{2-10}
 \multicolumn{1}{c|}{ }& lags & 1 & 2 & 3 & 4 & 5 & 20 & 40 & 60 \\
\hline
\multirow{4}{*}{\begin{turn}{90}Classical\end{turn}}
& $n=100$ &  79.84 &78.88 &78.16 &74.56 &74.68 &20.46  &0.56  &5.58\\ \cline{2-10}
& $n=200$ &  97.38 &97.80 &97.62 &97.36 &97.64 &90.44 &46.00  &1.52\\ \cline{2-10}
& $n=400$ &  99.98  &99.96 &100.00 &100.00 &100.00 &99.96 &99.72 &97.06\\ \cline{2-10}
& $n=800$ &  100.00 &100.00 &100.00 &100.00 &100.00 &100.00 &100.00 &100.00\\ \hline\hline
\multirow{4}{*}{\begin{turn}{90}RP\end{turn}}
& $n=100$ &  4.86 &3.98 &4.30 &4.46 &4.30 &3.76 &3.24 &3.22\\ \cline{2-10}
& $n=200$ &  4.44 &4.40 &4.56 &4.28 &4.40 &4.36 &3.84 &4.78\\ \cline{2-10}
& $n=400$ &  5.44 &4.90 &4.84 &4.84 &4.76 &4.90 &4.74 &4.28\\ \cline{2-10}
& $n=800$ &  6.86 &6.48 &6.38 &6.24 &6.48 &6.74 &6.20 &6.00\\ \hline\hline
\multirow{4}{*}{\begin{turn}{90}RPV\end{turn}}
& $n=100$ &  6.36 &6.24 &6.78 &7.38 &6.30 &4.28 &2.58 &1.10\\ \cline{2-10}
& $n=200$ &  5.88 &6.46 &6.16 &6.06 &6.34 &4.68 &4.38 &4.98\\ \cline{2-10}
& $n=400$ &  5.32 &5.58 &5.50 &5.98 &5.76 &5.26 &5.02 &4.20\\ \cline{2-10}
& $n=800$ &  5.40 &5.52 &5.74 &5.42 &5.54 &6.22 &5.26 &5.64\\ \hline
\end{tabular}
\label{conf-bound-shift}
\end{center}
\end{table}


\begin{table}[hh]\!\!\!\!\!\!\!\!\!\!
\begin{center}
\caption{\small{Empirical size (in \%) of the adaptive and classical portmanteau tests ($\delta=1$, $m=5$ and no second order dynamics), obtained from $R=5000$ independent replications. The nominal asymptotic level of the tests is 5\%. The variance and zero returns probability are constant (the (a)-(1) case), the variance is constant and the zero returns probability is time-varying (the (a)-(2) case), and the variance and probability are both non-constant (the (c)-(2) case).}}
\begin{tabular}{|c|c||c|c|c|c|}
\hline
 \multicolumn{2}{|c||}{sample size} & 100 & 200 & 400 & 800 \\
\hline
\multirow{2}{*}{\begin{turn}{90} (a)-(1)\:\:\:\:\:\:\end{turn}}
& Classical & 4.20 & 5.08 & 4.00 & 4.60\\ \cline{2-6}
& RP & 5.68 & 5.54  & 4.64 & 4.92 \\ \cline{2-6}
& RPV & 5.76 & 5.92 & 4.72 & 4.74 \\ \hline\hline
\multirow{2}{*}{\begin{turn}{90}(a)-(2)\:\:\:\:\:\:\end{turn}}
& Classical & 72.62 & 95.72 & 99.94 & 100.00 \\ \cline{2-6}
& RP & 5.56 & 5.98 & 5.10& 4.88 \\ \cline{2-6}
& RPV & 6.9 & 6.60 & 5.80 & 5.62 \\ \hline\hline
\multirow{2}{*}{\begin{turn}{90}(c)-(2)\:\:\:\:\:\:\end{turn}}
& Classical & 98.70 & 100.0 & 100.0 & 100.0 \\ \cline{2-6}
& RP & 5.4 & 5.5 & 6.94 & 10.18\\ \cline{2-6}
& RPV & 7.94 & 6.68 & 5.84 & 5.66 \\ \hline
\end{tabular}
\label{port-reject}
\end{center}
\end{table}

\begin{table}[hh]\!\!\!\!\!\!\!\!\!\!
\begin{center}
\caption{\small{The Box-Pierce test p-values (in \%) for illiquid Santiago financial market stocks ($m=5$ and $\delta=1$). }}
\begin{tabular}{c|c|c|c|c|c|c|}\cline{2-5}
   & sample size & Classical & RP & RPV \\
  \hline
  \multicolumn{1}{|c|}{{\small Molymet}} & 4896 & 0.00 & 0.00 & \underline{\textbf{0.20}} \\
  \hline
  \multicolumn{1}{|c|}{{\small Las Condes}} & 1958 & \underline{\textbf{0.00}} & 0.00 & 0.55 \\
  \hline
  \multicolumn{1}{|c|}{{\small Cruzados}} & 2578 &0.00 & \underline{\textbf{0.75}} & 6.50 \\
  \hline
  \multicolumn{1}{|c|}{{\small Provida }} & 5179 & 0.00 & \underline{\textbf{0.00}} & 0.00 \\
  \hline
  \multicolumn{1}{|c|}{{\small Lipigas }} & 892 & \underline{\textbf{0.00}} & 0.15 & 6.48 \\
  \hline
  \multicolumn{1}{|c|}{{\small Conchatoro }} & 5188 & 0.00 & \underline{\textbf{0.00}} & 0.00 \\
  \hline
\end{tabular}
\label{stat-test}
\end{center}
\end{table}


\clearpage
\begin{figure}[p!]
\protect \includegraphics{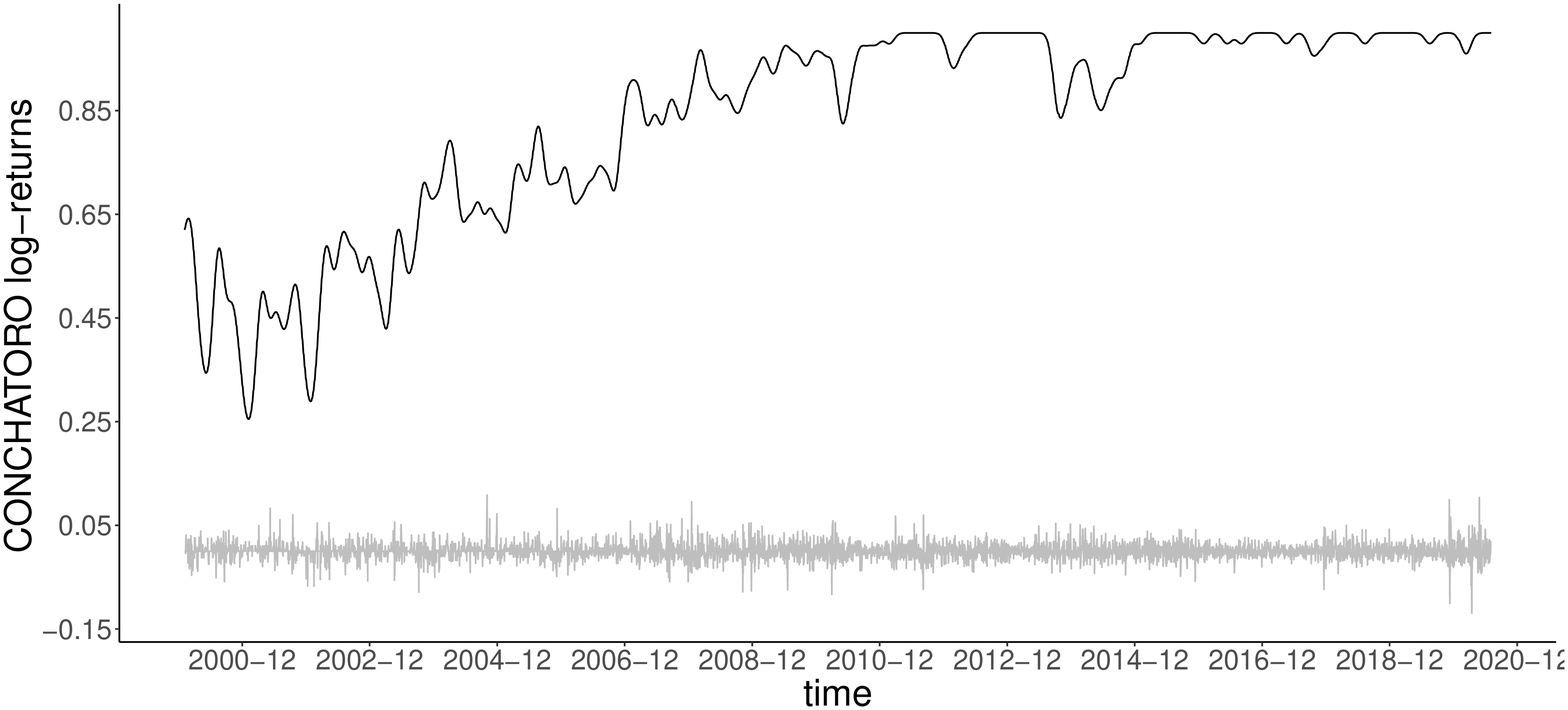}
\protect \includegraphics{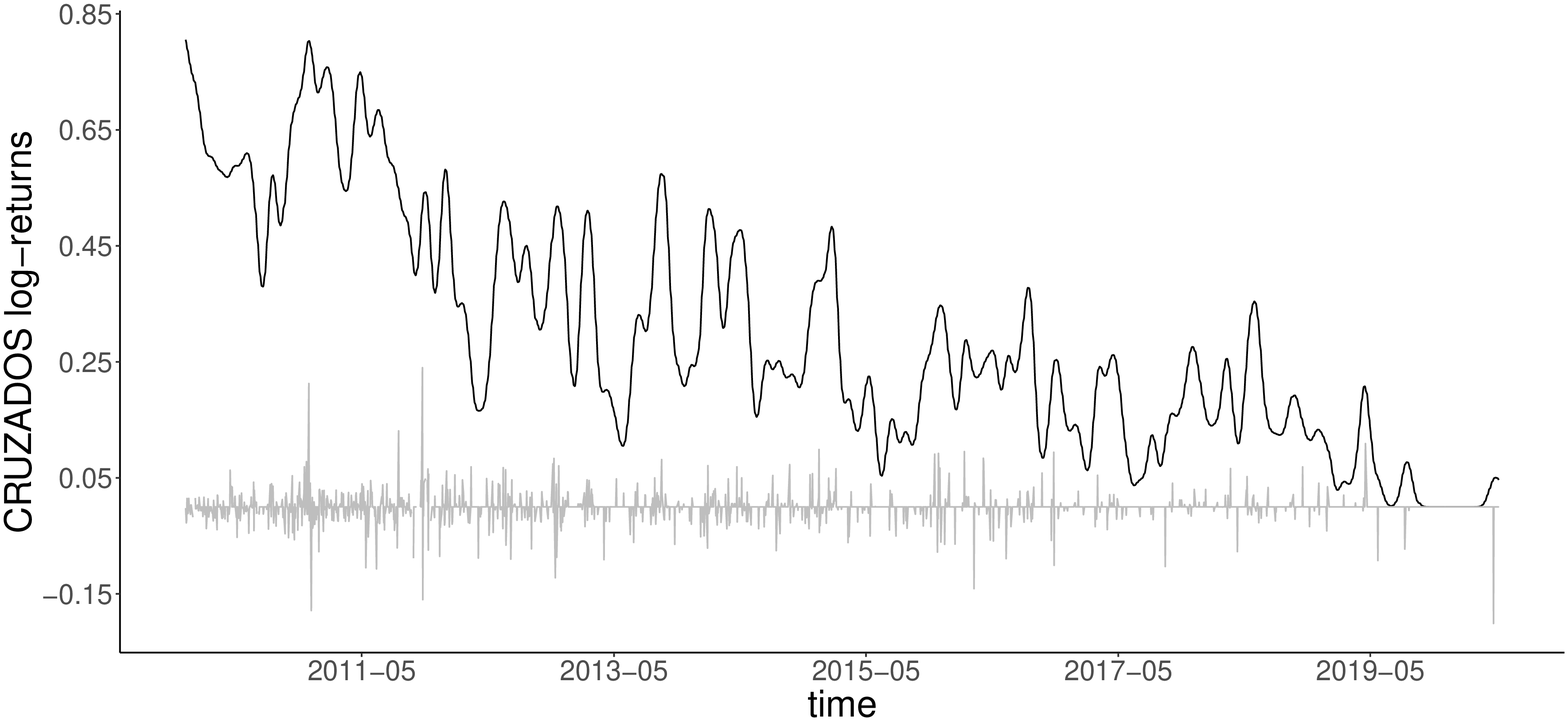}
\protect \includegraphics{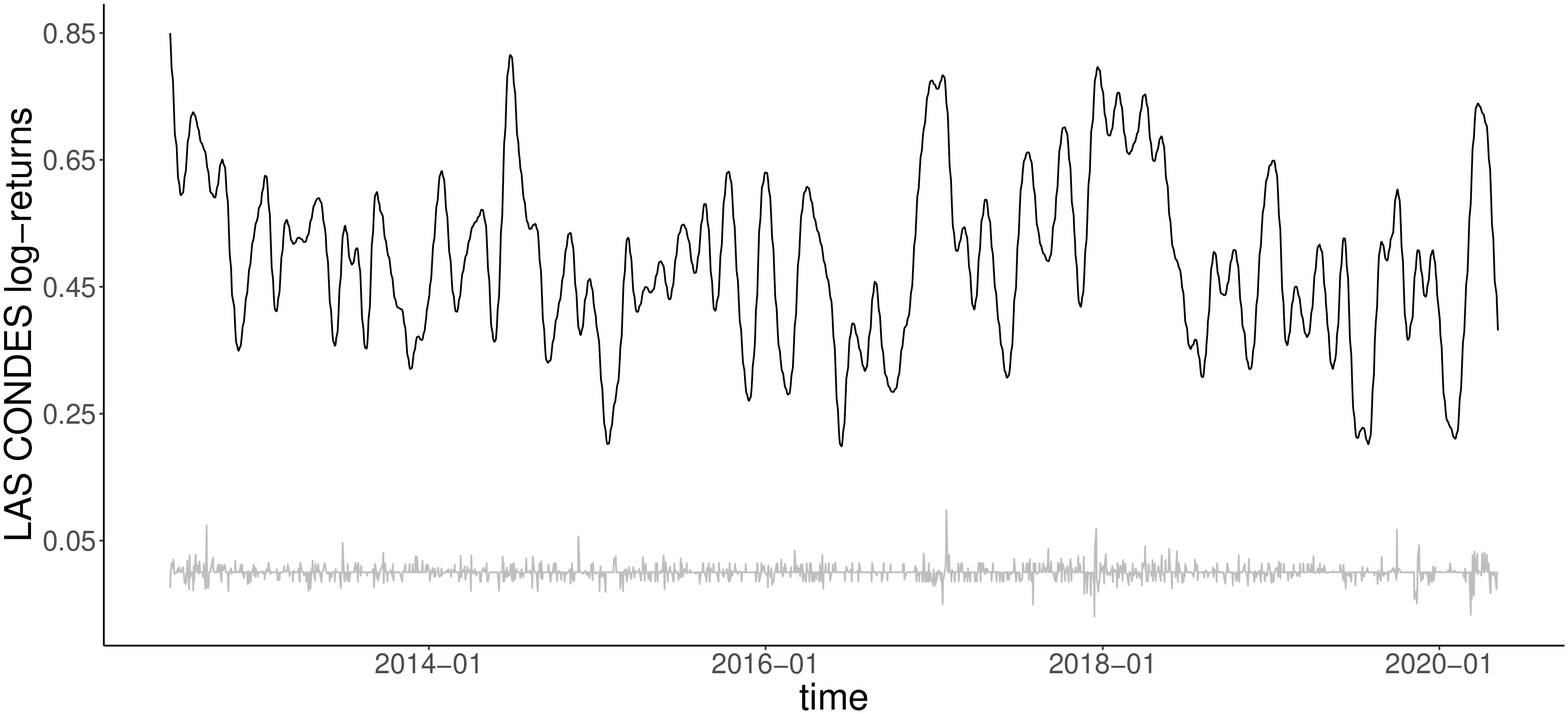}
\protect \includegraphics{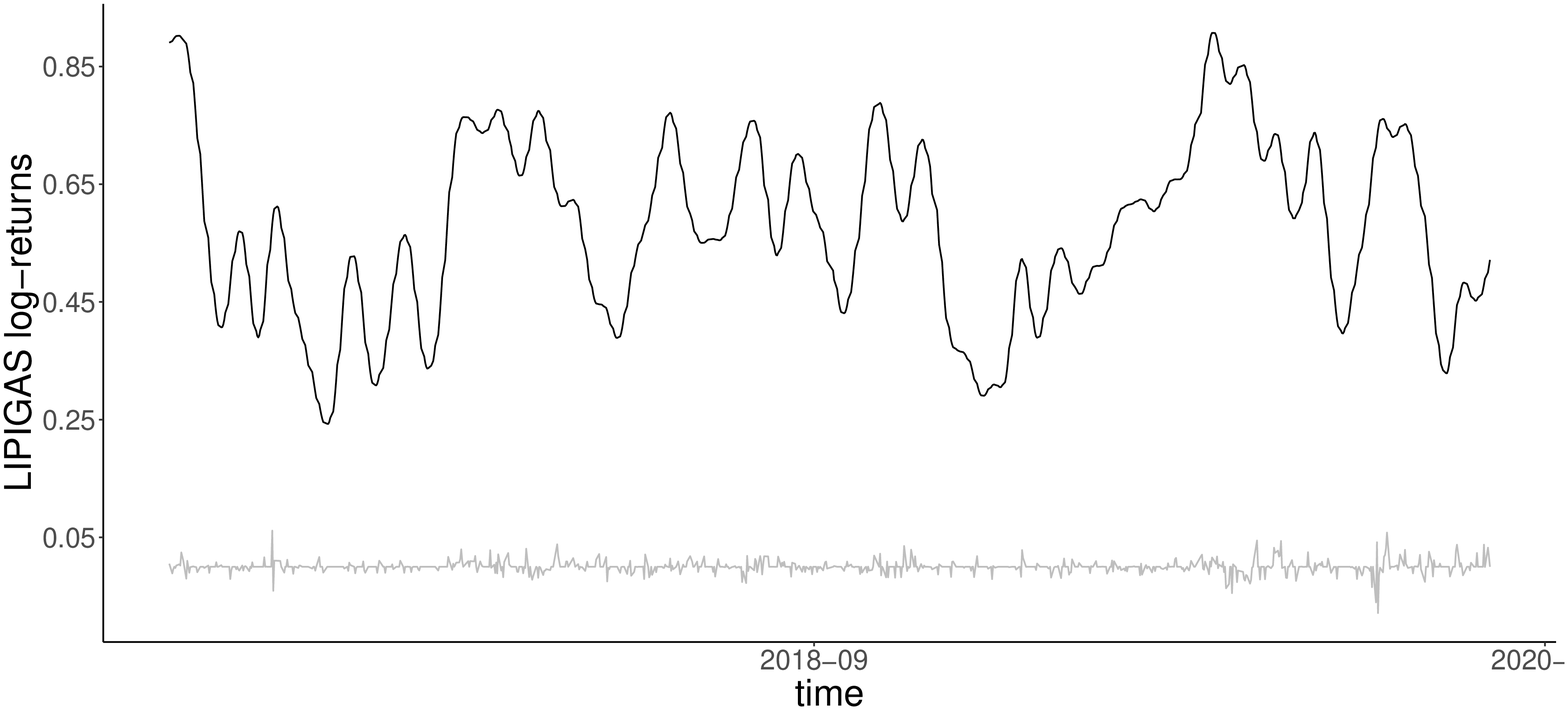}
\protect \includegraphics{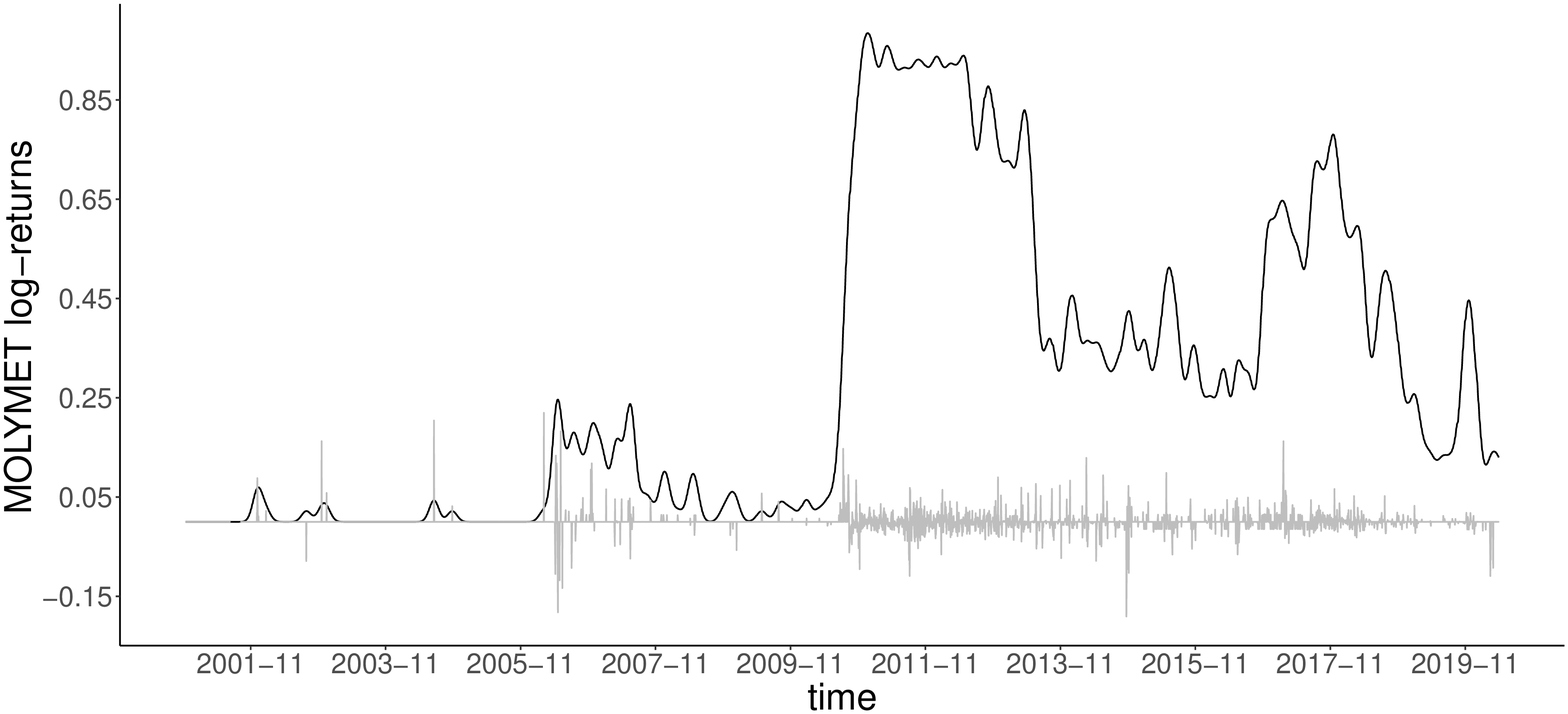}
\protect \includegraphics{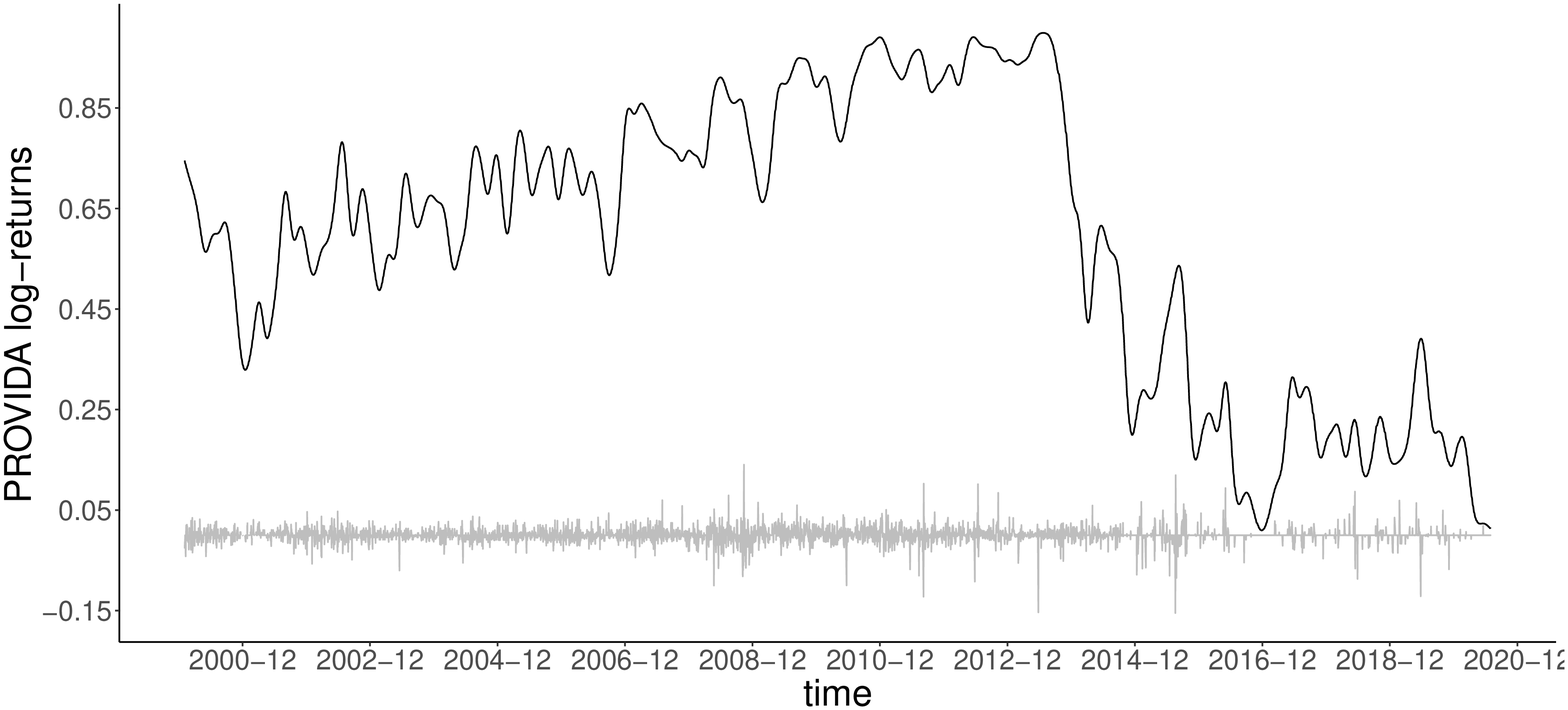}
\caption{\label{ns-stocks}
{\footnotesize The log-returns of different illiquid stocks of the Santiago financial market. The stocks seem to have a non stationary daily zero returns probability. The kernel smoothing estimator of $P(r_t\neq0)$ is displayed in full line. The bandwidth minimizes the CV criterion in (\ref{cv-px}) within a grid a values, and is equal to $b=0.0201$ for the Lipigas stock, and $b=0.0101$ for the Las Condes, Molymet, Conchatoro, Cruzados and Provida stocks. Data source: Yahoo Finance.}}
\end{figure}

\clearpage
\begin{figure}[h]
\begin{center}
 \includegraphics[scale=0.46]{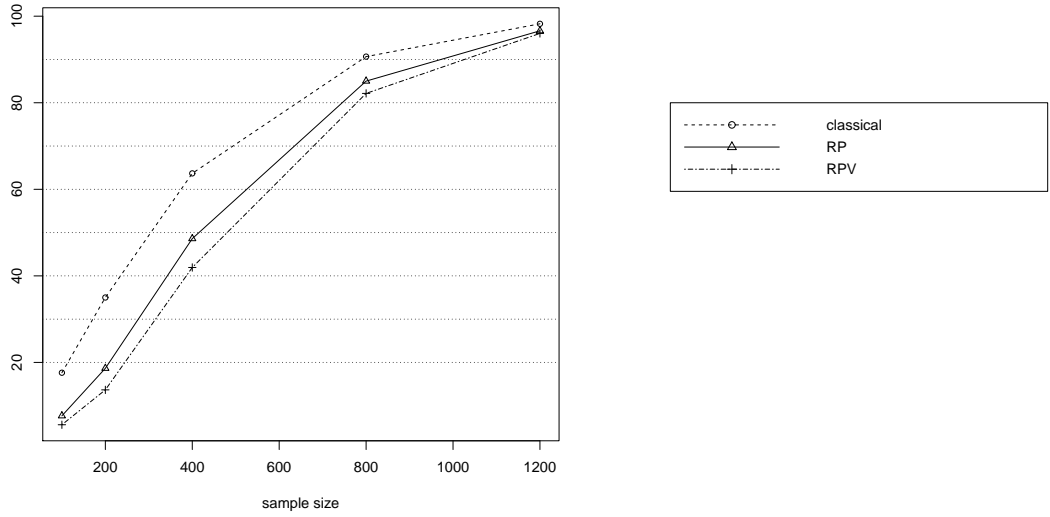}
\end{center}
\caption{The empirical power of the classical, RP and RPV tests. The data are simulated according to the stationary GARCH(1,1) model given in \eqref{garch-mod}.}
\label{fig-index}
\end{figure}

\clearpage
\vspace*{20cm}
\begin{figure}[h]\!\!\!\!\!\!\!\!\!\!
\protect \includegraphics{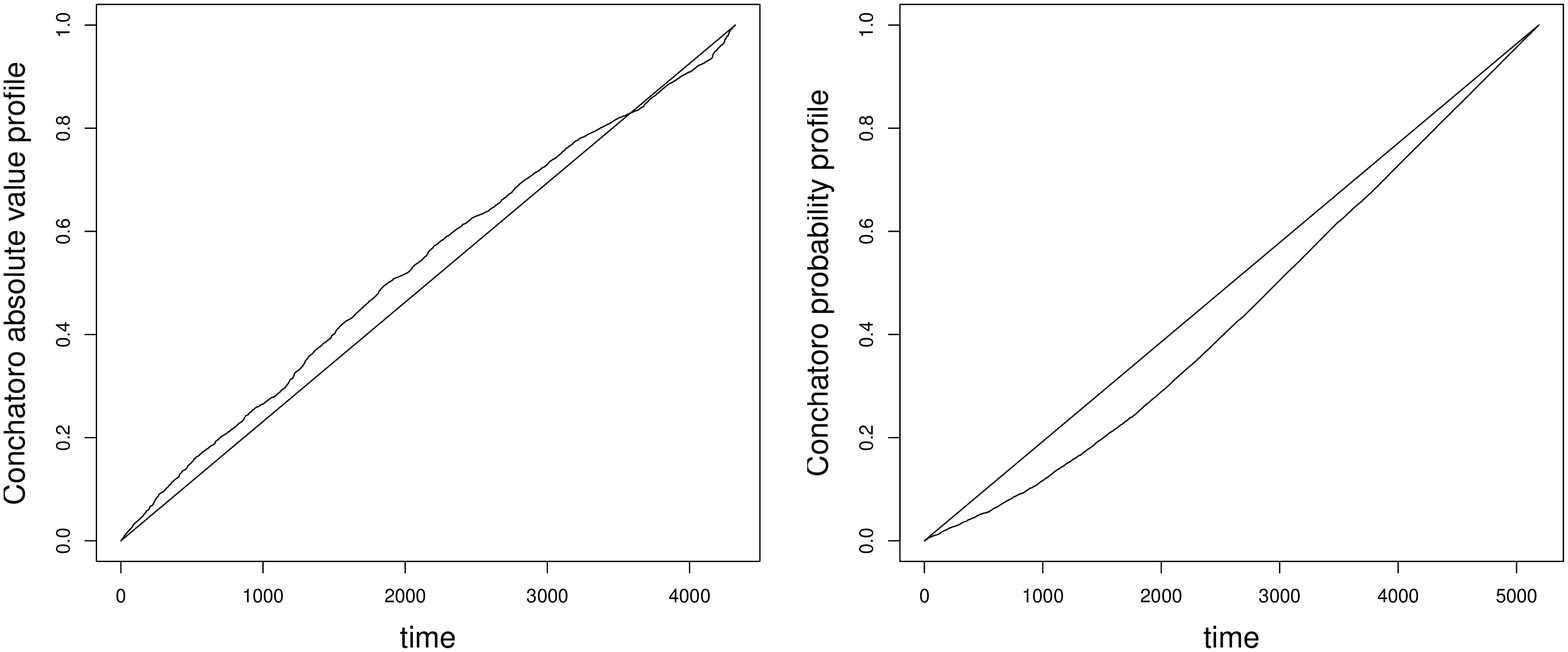}
\protect \includegraphics{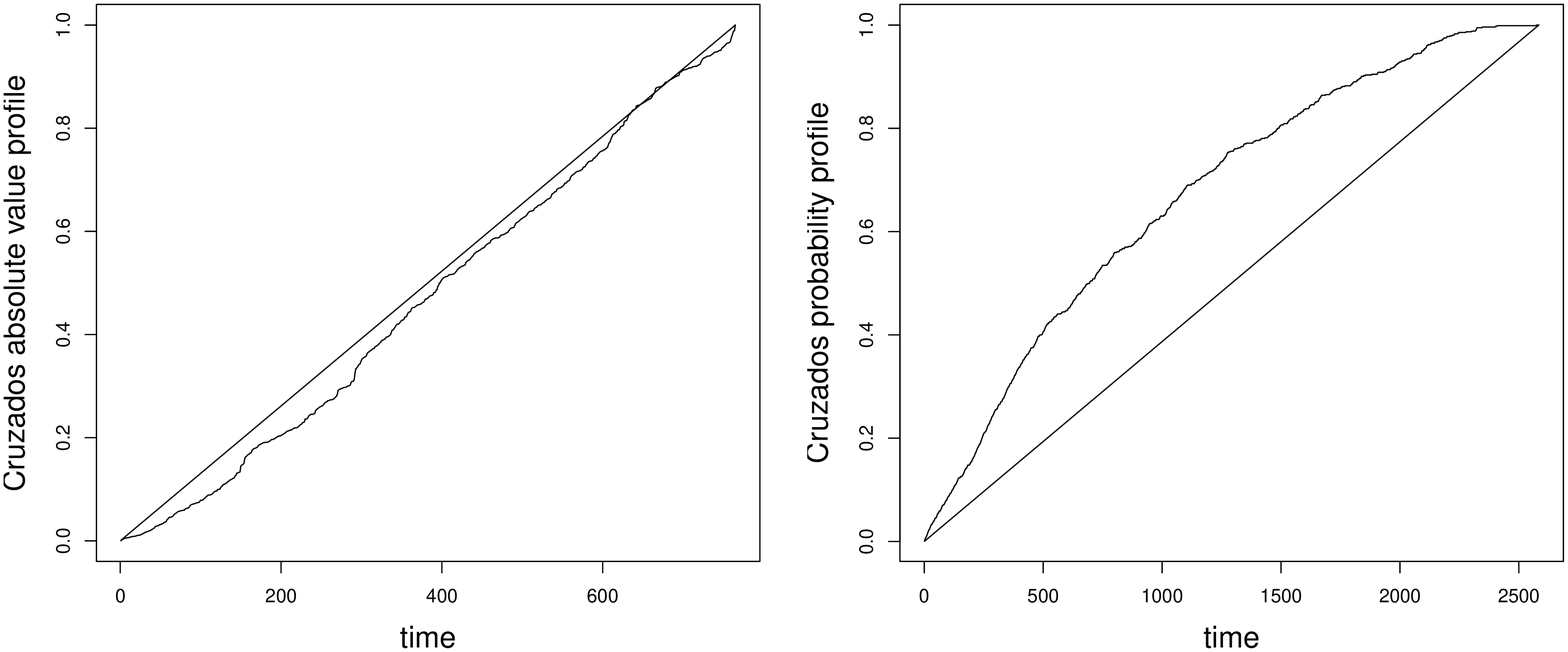}
\protect \includegraphics{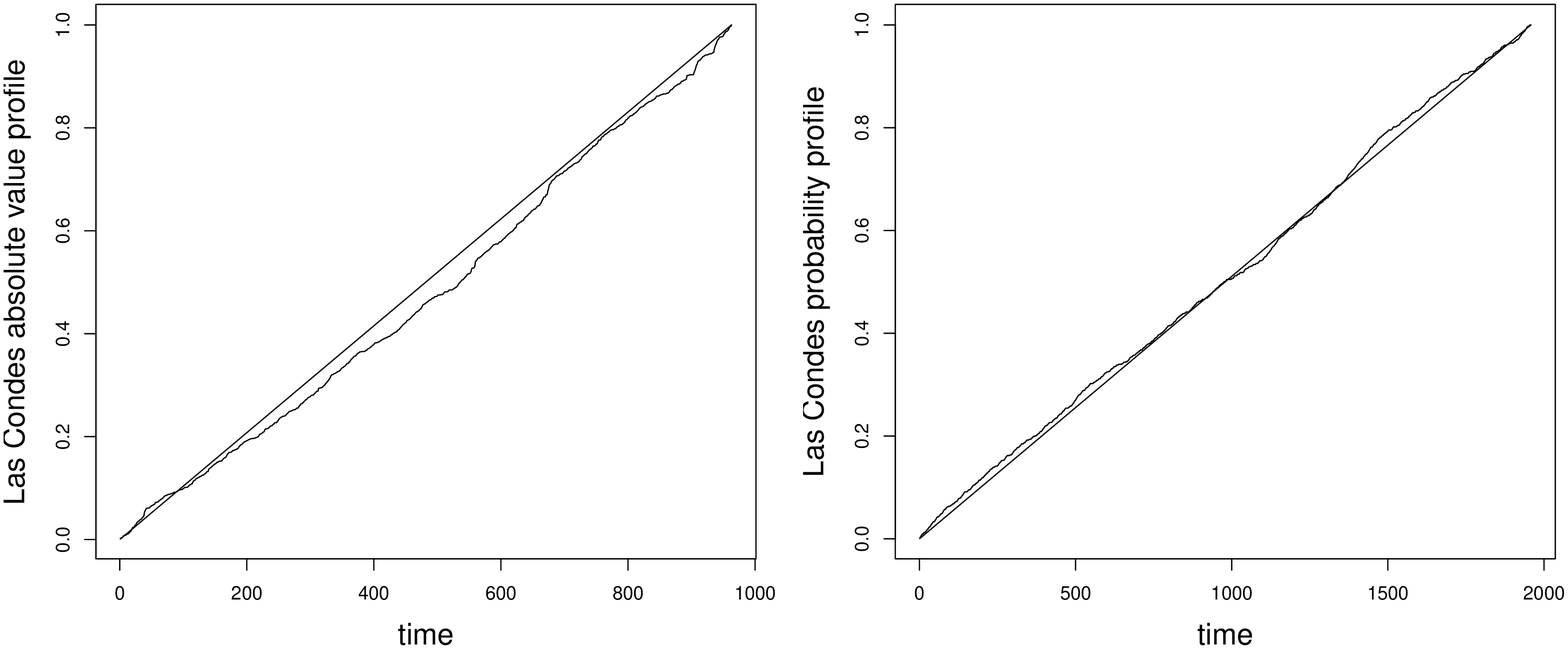}
\protect \includegraphics{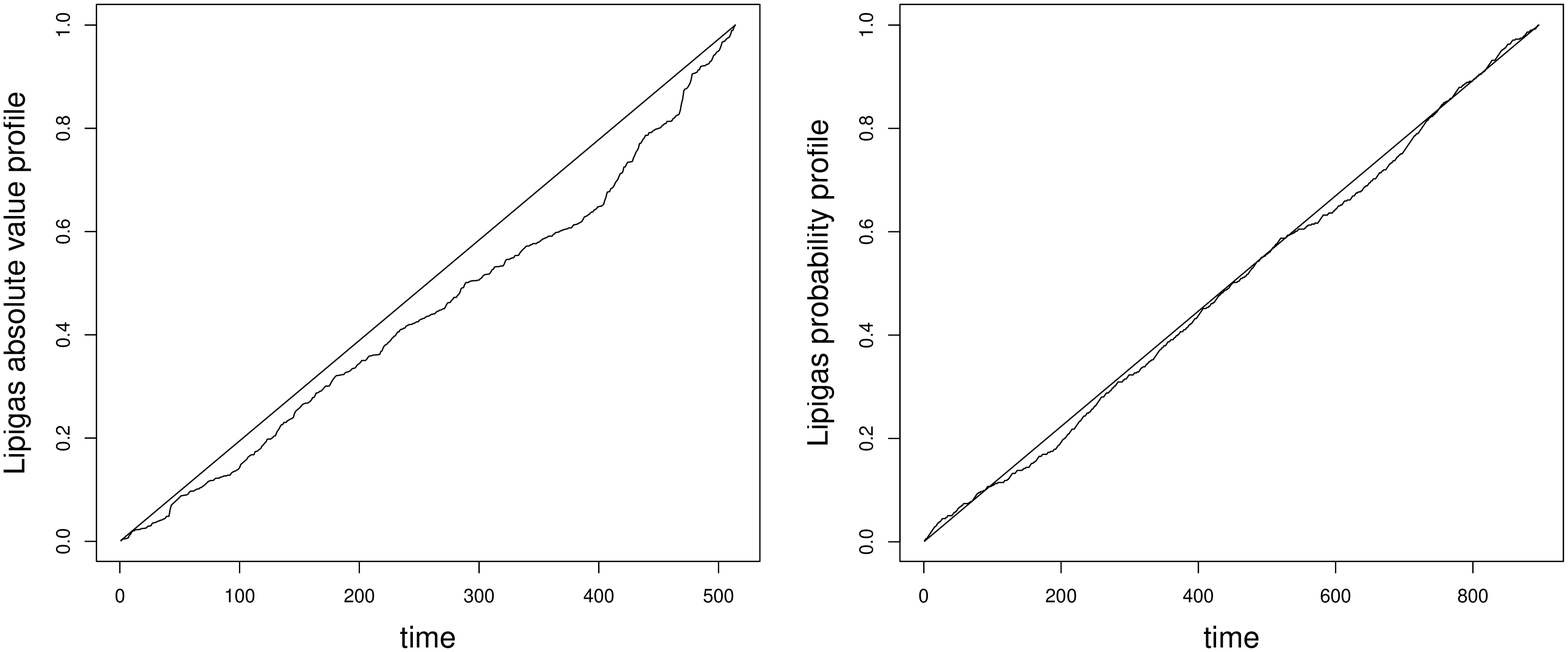}
\protect \includegraphics{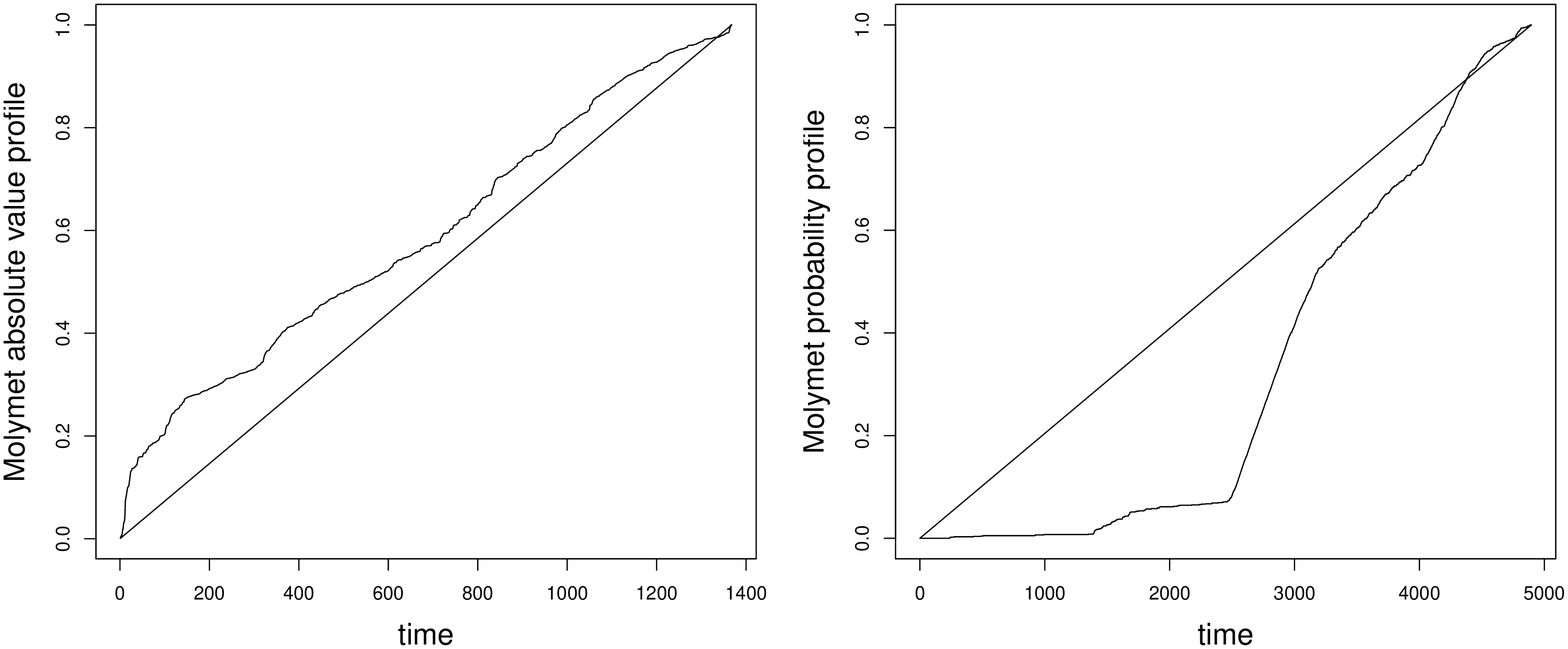}
\protect \includegraphics{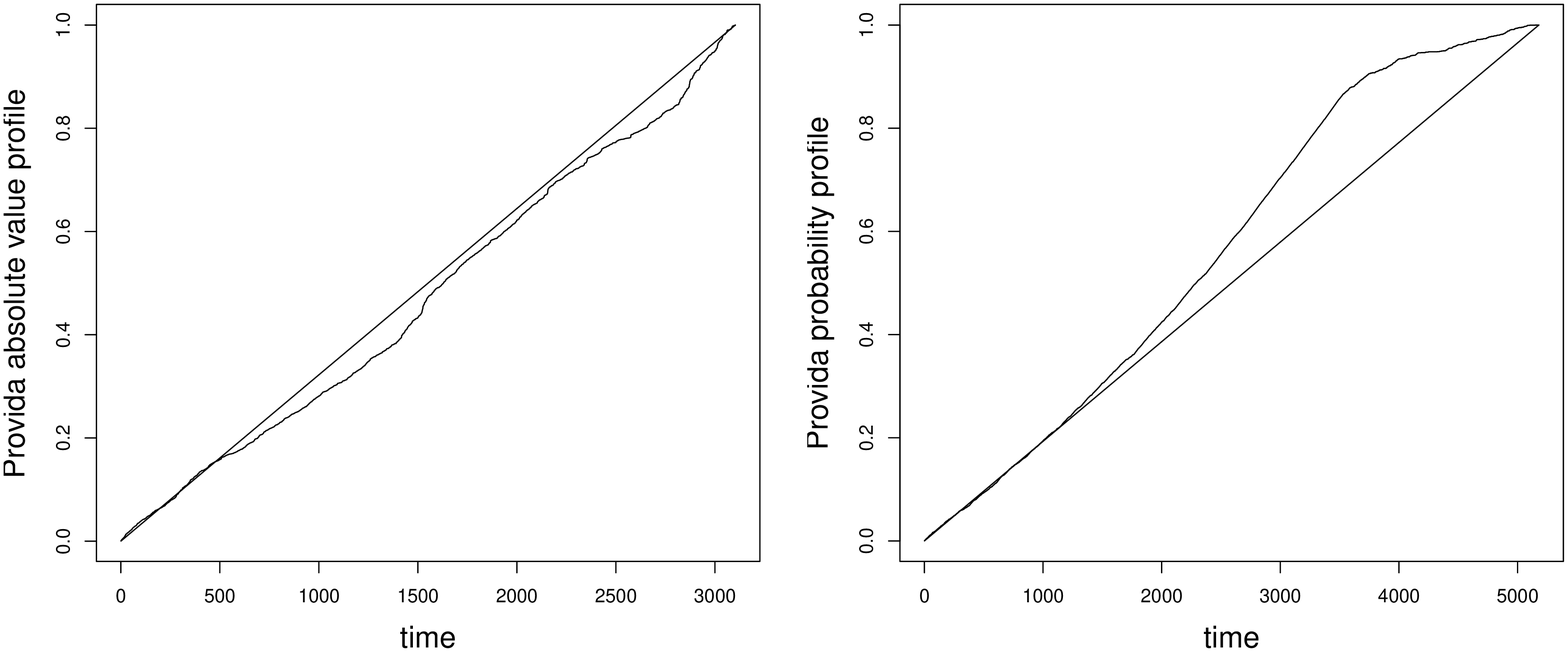}
\caption{\label{ns-stocks-profiles}
{\footnotesize The probability and absolute returns profiles of the studied stocks. The $ar(\cdot)$ and $p(\cdot)$ are to be compared with the identity function defined on $(0,1]$.}}
\end{figure}


\clearpage
\begin{figure}[h]
\begin{center}
 \includegraphics[scale=0.46]{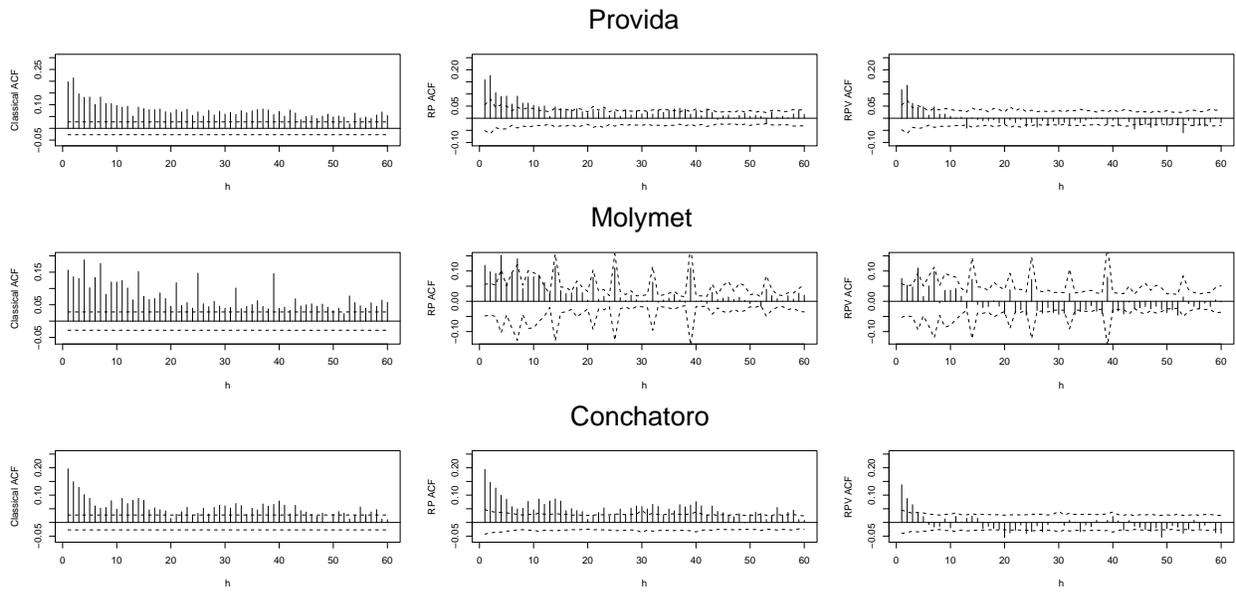}
\end{center}
\caption{The classical (left column), RPV and RP (middle and right columns) absolute returns autocorrelations ($\delta=1$) of the Provida, Molymet and Conchatoro stocks studied in the paper for $h=1,\dots,60$. The dashed lines correspond to the bootstrap and classical 95\% confidence bands.}
\label{fig-index-2}
\end{figure}

\clearpage
\begin{figure}[h]
\begin{center}
 \includegraphics[scale=0.46]{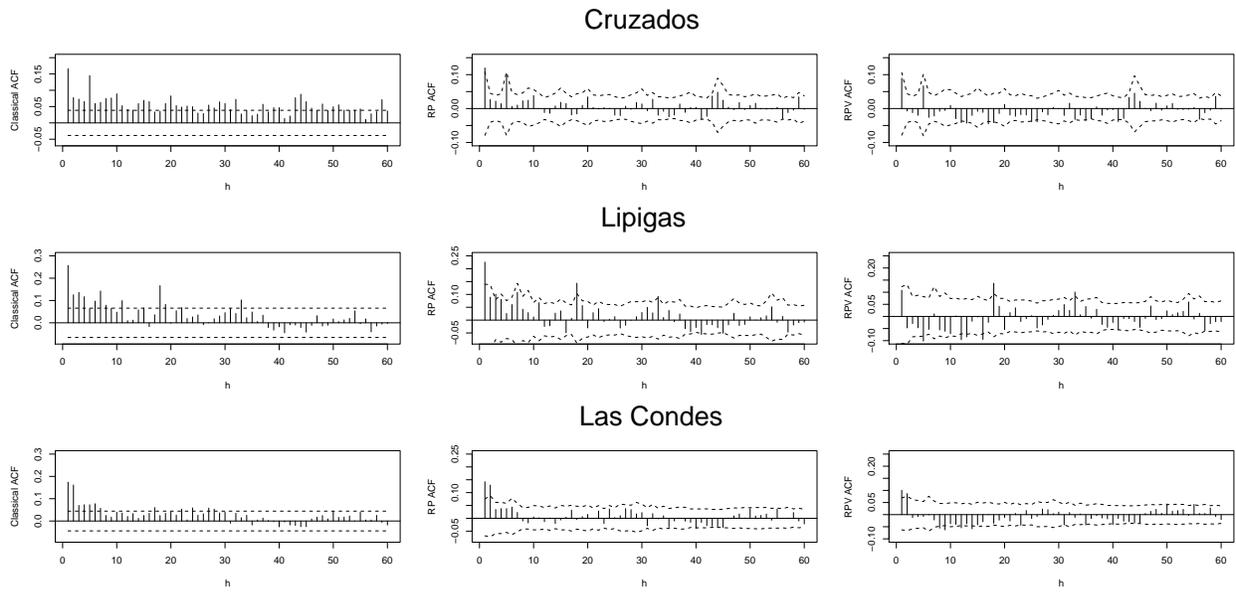}
\end{center}
\caption{The classical (left column), RPV and RP (middle and right columns) absolute returns autocorrelations ($\delta=1$) of the Cruzados, Lipigas, Las Condes stocks studied in the paper for $h=1,\dots,60$. The dashed lines correspond to the bootstrap and classical 95\% confidence bands.}
\label{fig-index-3}
\end{figure}

\end{document}